\documentclass[12pt,a4paper,oneside]{amsart}
\usepackage[english]{babel}
\usepackage[utf8]{inputenc}
\usepackage{amsmath,amssymb}
\usepackage{amsthm,mathtools}
\usepackage[foot]{amsaddr}
\usepackage{graphicx}
\usepackage{bbm}
\usepackage{geometry}
\usepackage{tikz}
\usepackage{tikz-cd}
\usepackage{mathrsfs}
\usepackage[colorinlistoftodos, textwidth=25mm]{todonotes}
\usepackage{mycommands}
\usepackage[pdfborder={0 0 0.5}]{hyperref}
\usepackage{verbatim}
\usepackage{enumitem}
\setlist[enumerate]{label=\arabic*.}

\theoremstyle{plain}
\newtheorem{thm}{Theorem}[section]
\newtheorem{lem}[thm]{Lemma}
\newtheorem{prop}[thm]{Proposition}
\newtheorem{exa}[thm]{Example}
\newtheorem{cor}[thm]{Corollary}
\newtheorem{conj}[thm]{Conjecture}
\newtheorem{problem}[thm]{Problem}
\theoremstyle{definition}
\newtheorem{defn}[thm]{Definition}
\theoremstyle{remark}
\newtheorem*{rmk}{Remark}

\usepackage{cleveref}
\crefname{section}{§}{§§}
\Crefname{section}{§}{§§}

\makeatletter
\renewcommand{\email}[2][]{%
  \ifx\emails\@empty\relax\else{\g@addto@macro\emails{,\space}}\fi%
  \@ifnotempty{#1}{\g@addto@macro\emails{\textrm{(#1)}\space}}%
  \g@addto@macro\emails{#2}%
}
\makeatother

\begin{document}
\title{Quaternionic Mahler measure}
\author{\small Weijia Wang}
\author{\small Hao Zhang}
\address[Weijia Wang]{School of Mathematics, Shandong University, Jinan 250100, P. R. China}
\address[Hao Zhang]{School of Mathematics, Hunan University, Changsha 410082, P. R. China}
\email[Weijia Wang]{weijiawang@amss.ac.cn}
\email[Hao Zhang (Corresponding author) ]{zhanghaomath@hnu.edu.cn}
\date{}


\begin{abstract}
    We introduce the quaternionic Mahler measure, extending the classical complex Mahler measure to non-commutative polynomials. We establish the existence of quaternionic Mahler measure for slice regular polynomials in one and two variables. We study the quaternionic Mahler measure for real and slice regular polynomials, and consider the associated Lehmer problem. Various formulas of quaternionic Mahler measures are proved.
\end{abstract}

\maketitle

\section{Introduction}
Given a non-zero complex Laurent polynomial $P\in\C[x_1^{\pm},\dots,x_n^{\pm}]$, the complex (logarithmic) Mahler measure of $P$ is defined as
\[\mC(P)\coloneqq\int_{\mbt^n(\C)}\log\abs{P(x_1,\dots,x_n)}\dif\mu(\mbt^n(\C)),\]
where $\mbt^n(\C)=\{(x_1,\dots,x_n)\in\C^{n}\,|\,|x_1|=\dots=|x_n|=1\}$ is the unit complex torus and  $d\mu(\mbt^n(\C))=(2\pi i)^{-n}\prod_{\ell=1}^{n}\frac{\dif x_{\ell}}{x_\ell}$ is its probability Haar measure.

The origins of Mahler measure can date back to the early 20th century when Lehmer~\cite{Leh} introduced it (of one variable) in his study of finding large prime numbers. Lehmer proposed what came to be known as \emph{Lehmer's conjecture}, which concerns the smallest non-zero Mahler measure that an irreducible polynomial in $\Z[x]$ could attain. Later, Mahler~\cite{Mah} extended the concept to several complex variables in utilization to explore the heights of complex commutative polynomials. 

In the early 1980s, Smyth~\cite{Smy} discovered the notable identity,
\[\mC(1+x+y)=\frac{3\sqrt{3}}{4\pi}L(\chi_{-3},2)=L'(\chi_{-3},-1),\]
where $L(\chi_{-3},s)$ is the Dirichlet $L$-function of the odd Dirichlet character of conductor $3$. The machinery of relating the Mahler measure with $L$-functions and regulators, continues to be explored through the work of Boyd~\cite{Boy}, Deninger~\cite{Den} and Rodriguez-Villegas~\cite{RV}. For a thorough survey about complex Mahler measure, the readers are also recommended to consult the book~\cite{B-Z} of Brunault and Zudilin.

Variants of Mahler measure have also been proposed in the literature. For example, the definition of elliptic Mahler measure by Everest and n\'{\i} Fhlath\'uin~\cite{EnF}, Akatsuka's Zeta Mahler measure~\cite{Aka}, and Pritsker's examination of the areal Mahler measure~\cite{Pri1} by replacing the integration with area measure on the open unit complex disk (see also Pritsker~\cite{Pri2}, Lal\'{\i}n--Roy~\cite{LR1,LR2}).

It is natural to seek an analogue of Mahler measure for non-commutative polynomials. The quaternions are perhaps the most familiar non-commutative algebra, and hence a natural starting point for such an extension.

In this manuscript, let $\mbh$ be the quaternion algebra. For a non-commutative polynomial $P(x_1,\dots,x_n)$ with coefficients in $\mbh$, we consider the following quaternionic analogue of Mahler measure
\[\mH(P)\coloneqq\int_{\mbt^n(\mbh)}\log\abs{P(x_1,\dots,x_n)}\dif\mu(\mbt^n(\mbh)),\]
where $\mbt^n(\mbh)=\{(x_1,\dots,x_n)\in\mbh^{n}\,|\,|x_1|=\dots=|x_n|=1\}$ becomes the unit quaternion torus and $d\mu(\mbt^n(\mbh))$ is again the probability Haar measure. Following~\cite{Aka}, we may also define the corresponding quaternionic zeta Mahler measure of $P(x_1,\dots,x_n)$ by
\[Z_{\mbh}(P;s)\coloneqq\int_{\mbt^n(\mbh)}\abs{P(x_1,\dots,x_n)}^s\dif\mu(\mbt^n(\mbh)),\qquad s\in \mbc.\]

The quaternionic Mahler measure shares many properties analogous to its complex counterpart. For example, given two quaternionic polynomials $P$ and $Q$, one has $\mH(PQ)=\mH(P)+\mH(Q)$, where $PQ$ must be the \emph{non-commutative product} of $P$ and $Q$. It is worth noting that even for slice-preserving (i.e. real) polynomials, the quaternionic Mahler measure is not merely the sum of Mahler measures of monomials. 

Our first result on monomials and slice-preserving polynomials is summarized as follows.
\begin{thm}\label{thm:Jensenformulas}
Let $\alpha\in \mbh$. If $\alpha$ is non-zero, set $\phi=\arccos(\RE \alpha/|\alpha|)$; if $\alpha=0$, set $\phi=0$.
\begin{enumerate}
    \item For $\mH(x-\alpha)$, we have 
    \[m_{\mbh}(x-\alpha)=\begin{dcases}
        \frac{1}{4}\,|\alpha|^2&\quad \text{if } |\alpha|\leq 1,\\
        \frac{1}{4}\,|\alpha|^{-2}+\log |\alpha|&\quad \text{if } |\alpha|>1.
    \end{dcases}\]
    \item For $m_{\mbh}(x^2-2\RE \alpha\cdot x+|\alpha|^2)$, we have 
    \[m_{\mbh}(x^2-2\RE \alpha\cdot x+|\alpha|^2)=\begin{dcases}
        \frac{1}{2}\,|\alpha|^{2}\cos 2\phi &\text{if }|\alpha|\leq 1,\\
        \frac{1}{2}\,|\alpha|^{-2}\cos 2\phi +2\log|\alpha|&\text{if }|\alpha|> 1.
    \end{dcases}\]
    \item  For $\mH(P(x))$, where $P(x)$ is a real polynomial completely factorized as
    \begin{equation}\label{eq:factorofp}
        P(x)=c\,\prod (x-\alpha_h)^{n_h}\prod (x^2-2\RE \beta_{\ell}\cdot x+\abs{\beta_{\ell}}^2)^{m_\ell},
    \end{equation}
we have
\begin{align*}
        \mH(P(x)) &= \log \abs{c}+\sum n_{h}\mH(x-\alpha_h)\\
        &\quad + \sum m_{\ell}\mH(x^2-2\RE \beta_{\ell}\cdot x+\abs{\beta_{\ell}}^2).
    \end{align*}
\end{enumerate}  
\end{thm}
\begin{rmk}
Our results here are consistent with, and can be derived from \cite{AB}, where a Jensen formula for slice-preserving regular functions is established.
\end{rmk}

In addition, we calculate the zeta Mahler measure of quaternionic monomials
\[Z_{\mbh}(x-\alpha;s)\coloneqq\int_{\mbt^{1}(\mbh)}\abs{x-\alpha}^s\dif\mu(\mbt^{1}(\mbh)),\]
as listed in Theorem~\ref{thm:zetameaofx}.

Additionally, we consider the Mahler measure of a class of more general quaternionic univariate polynomials, called \emph{slice regular polynomials}. These polynomials are a special kind of entire \emph{slice regular functions}. The theory of slice regular functions, which extends the theory of holomorphic functions in complex variables, has been systematically developed over the past two decades. Initially introduced within the context of quaternion algebra by Gentili and Struppa~\cite{GS1} in 2006, slice regular theory has been successfully broadened to the real alternative $\ast$-algebras~\cite{GP1} and several variables~\cite{GP2}. (See also the monograph~\cite{GSS} for an overview on the theory of slice regular functions in one quaternionic variable).

For a slice regular polynomial of one variable, we prove the existence of its Mahler measure and establish also one lower bound.
\begin{thm}\label{thm:existence}
    Let $P(x)$ be a non-zero slice regular quaternionic polynomial. Then the quaternionic Mahler measure $\mH(P)$ exists as an improper Riemann integral. Moreover, we have always
    \[\mH(P)\geq\frac{1}{2}\mH(P^{s}),\] 
    where  the equality holds if and only if $P$ is a slice-preserving polynomial.
\end{thm}
This theorem has also been extended to slice regular polynomials in two-variable cases, see Section~\ref{sec7}.

Our next results concern the \emph{$\ast$-products} of quaternionic monomials (see Section~\ref{sec:sliceregular} for definition).
\begin{thm}\label{thm:iteratedi}
Let $n\geq1$ and let $(x-i)^{*n}$ denote the iterated $\ast$-product. Then we have
    \[m_{\mbh}((x-i)^{*n})=\frac{2n}{\pi}\int_0^{\infty}t\coth nt\tanh^2t\sech t\dif t-\frac{2+n}{4}.\]
\end{thm}
In particular, this leads to the following
\[m_{\mbh}(x^2-2xi-1)=\frac{5}{6},\]
\[m_{\mbh}(x^3-3x^2i-3x+i)=\frac{31}{4}-2\pi.\]

In the following theorem, we examine examples of Mahler measures for non-commutative linear polynomials in the Sylvester form.
\begin{thm}\label{thm:ab}
    Let $a,b\in \mbh$ be non-zero and let $L(x)=ax+xb$ be not identically zero. Then 
    \[m_{\mbh}(L)=\frac{1}{2}\left(\frac{\lambda^{+}\log\lambda^{+}-\lambda^{-}\log \lambda^{-}}{\lambda^{+}-\lambda^{-}}-1\right),\]
    where we put   
$
\lambda^{\pm}=(\RE a+\RE b)^2+(|\!\IM a|\pm|\!\IM b|)^2.
$
\end{thm}

Additionally, we provide several instances of Mahler measures in multiple variables. As for Smyth's example $1+x+y$, its quaternionic Mahler measure is as follows.

\begin{thm}\label{thm:Smyth}
    We have 
    \[m_{\mbh}(1+x+y)=\frac{1}{2}-\frac{11\sqrt{3}}{16\pi}+\frac{3\sqrt{3}}{4\pi}L(\chi_{-3},2).\]
\end{thm}

\begin{rmk}
  The formula here can be compared with Smyth's formula on $m_{\mathbb{C}}(1+x+y)$ and Lal\'{\i}n--Roy's~\cite{LR1} evaluation of the areal Mahler measure $m_{\mathbb{D}}(1+x+y)$. We have indeed the following relation
\[m_{\mathbb{H}}(1+x+y)-m_{\mathbb{D}}(1+x+y)=\frac{1}{3}.\]
The resemblance of the Jensen formulas of monomials in the complex case, the areal case and the quaternionic case might serve as a foundational explanation.
\end{rmk}

Generalizing the argument of Maillot~\cite{Mai} to the quaternionic setting, we also derive the following formula.

\begin{thm}\label{thm:Maillot}
    Let $a,b,c\in \mbh$ be non-zero. If $|a|,|b|,|c|$ do not form a triangle, then we have
    \[m_{\mbh}(ax+by+c)=\log\max\{|a|,|b|,|c|\}+\frac{|a|^2+|b|^2+|c|^2}{4\max\{|a|,|b|,|c|\}^2}-\frac{1}{4}.\]
\end{thm}

This article is structured as follows. Section \ref{sec2} presents preliminary properties on quaternions. In Section \ref{sec:sliceregular}, a concise introduction to quaternion slice regular functions is provided. The quaternion Mahler measure of real  polynomials is computed in Section \ref{sec4}, with a generalization to slice regular polynomials explored in Section \ref{sec:mahmeaofleftpol}. Section \ref{sec6} offers examples of non-commutative univariate polynomials, while Section \ref{sec7} provides some examples of multivariable quaternion Mahler measures. Finally, we introduce the quaternion analogue of the Lehmer problem in Section \ref{sec8}, along with numerical evidence.

\section*{Acknowledgements}
The authors thank Fran\c{c}ois Brunault and Matilde Lal\'{\i}n for helpful discussions. The authors are also grateful to the anonymous referee for a careful reading of the manuscript and for valuable comments and suggestions that improved the exposition.

The first author is supported by  China
Postdoctoral Science Foundation (Grant No. 2024M763477), National Natural Science Foundation of China (Grant No. 1250012812). The second author is supported by Fundamental Research Funds for the Central Universities (Grant No. 531118010622), National Natural Science Foundation of China (Grant No. 1240011979) and Hunan Provincial Natural Science Foundation of China (Grant No. 2024JJ6120). 

\section{Preliminaries on quaternions}\label{sec2}
The quaternion algebra $\mbh$ is a real non-commutative algebra with the canonical basis $\{1,i,j,k\}$. A quaternion number $q$ in $\mbh$ can be written uniquely in the form $q=x_0+ix_1+jx_2+kx_3$, where $x_\ell$ are real numbers and $i, j, k$ share the multiplication rule
\[i^2=j^2=k^2=-1,\quad ij=k=-ji,\quad jk=i=-kj,\quad ki=j=-ik.\] Every quaternion number $q$ consists of two parts, the \emph{real part} $\RE q=x_0$ and the \emph{imaginary part} $\IM q=ix_1+jx_2+kx_3$. The \emph{conjugate} of $q$ is $\overline{q}=\RE q-\IM q$ and the \emph{norm} of $q$ is defined as $\abs{q}^2=q\overline{q}=x_0^2+x_1^2+x_2^2+x_3^2$. Denote $\S$ by the unit $2$-sphere of purely imaginary quaternions, that is,
\[\S=\{q=ix_1+jx_2+kx_3\in\mbh\,|\,x_1^2+x_2^2+x_3^2=1\},\]
these are exactly the quaternions $I$ such that $I^2=-1$. Then every quaternion number can be rewritten as $q=\alpha+I\beta$ for some $\alpha,\beta\in\R$ and $I\in\S$. 
For a fixed $I\in\S$, the set of all quaternions in this form is defined as a \emph{slice}, denoted by $\C_{I}=\R+\R I$. Note that the quaternion algebra $\mbh$ can be decomposed as the union of all the slices $\C_{I}$ where $I\in\S$
\[\mbh=\bigcup_{I\in\S}\C_{I},\quad \R=\bigcap_{I\in\S}\C_{I}.\]

Two quaternions $p,q\in\mbh$ are called \emph{similar} if there exists a non-zero $r\in\mbh$ such that $q=r^{-1}pr$. It is well known that two quaternions $p,q$ are similar precisely if $\RE p=\RE q$ and $\abs{\!\IM p}=\abs{\!\IM q}$. The set of all quaternions that are similar to $p$ forms therefore a $2$-sphere given by
\[\S_{p}\coloneqq\{q\in\mbh\,|\, q\text{ is similar to } p\}=\RE p+\abs{\!\IM p}\,\S.\]

Let $I,J\in\S$ such that $I\perp J$.  Then every quaternion can be written uniquely as $q=z+Jw\in \mbh$ with $z,w\in \mbc_{I}$. This gives the isomorphism 
\[\mbt^1(\mbh)\simeq \SU(2)=\bigg\{\begin{pmatrix}
    z&w\\-\overline{w}&\overline{z}
\end{pmatrix}\,\bigg{|}\,z,w\in\C\text{ such that }|z|^2+|w|^2=1\bigg\}.\]
The probability Haar measure on $\SU(2)$ is given by $\frac{1}{4\pi^2}\frac{\dif z}{z}\wedge \dif w\wedge \dif \overline{w}$. Thus, the quaternionic Mahler measure may be expressed as
\begin{equation}\label{eq:SU(2)}
    \mH(P)=\frac{1}{(4\pi^2)^{n}}\int_{{\SU(2)}^n}\log\abs{P(x_1,\dots,x_n)}\,\prod_{\ell=1}^{n}\frac{\dif z_\ell}{z_\ell}\dif w_\ell\dif \overline{w_\ell},
\end{equation}
where each $x_\ell=\begin{psmallmatrix}
    z_\ell&w_{\ell}\\-\overline{w}_{\ell}&\overline{z}_\ell
\end{psmallmatrix}$ is in $\SU(2)$.

The set of unit quaternions $\mbt^1(\mbh)$ is also isomorphic to the $3$-sphere $S^{3}$. This yields the formula
\begin{equation}\label{eq:hyperspherical}
\mH(P)=\int_{{(S^{3})}^{n}}\log\abs{P(x_1,\dots,x_n)}\,\mathrm{d}^{n}\mu(S^{3}).
\end{equation}
In particular, given an orthonormal basis $\{I,J,K=IJ\}$ in $\S$, then every unit quaternion $q=x_0+ x_1I+x_2J+x_3K$ admits a hyperspherical parametrisation: there exist angles $\theta,\theta_1$ and $\theta_2$ such that
\[\begin{cases}
x_0=\cos\theta,\\
x_1=\sin\theta\cos\theta_1,\quad \\
x_2=\sin\theta\sin\theta_1\cos\theta_2,\\
x_3=\sin\theta\sin\theta_1\sin\theta_2.
\end{cases}
\quad
\theta\in[0,\pi],\ \theta_1\in[0,\pi],\ \theta_2\in[0,2\pi).
\]
The Haar measure of $S^{3}$ is given by 
\[\dif \mu(S^{3})=\frac{1}{2\pi^2}\sin^2\theta\sin\theta_1\dif \theta\dif \theta_1\dif \theta_2.\]

\section{Quaternion slice regular functions}\label{sec:sliceregular}
In the sequel, we offer here a minimalistic introduction for slice regular functions. We will concentrate on functions within the open ball $B(0,R)$ of radius $R>0$ in $\mbh$, while the general theory can be found in the references mentioned above. For each $I\in \mbs$, let $B_I:=B\cap \C_I=B(0,R)\cap \mbc_I$ denote the slices of the set $B=B(0,R)$.
\begin{defn}
    Let $f$ be a quaternion-valued function defined on $B=B(0,R)$. For each $I\in \mbs$, let $f_I=f|_{B_I}$ be the restriction of $f$ on $B_I$. Then $f$ is called slice (left) regular if $f_I$ is holomorphic for all $I\in \mbs$, i.e. it has continuous partial derivatives and 
    \[\overline{\partial}_I f(x+yI)=\frac{1}{2}\left(\frac{\pa}{\pa x}+I\frac{\pa}{\pa y}\right)f_I(x+yI)=0.\]
\end{defn}

As is proved in~\cite{GS1}, a slice regular function $f$ on $B(0,R)$ has a quaternion series representation
\[f(q)=\sum_{n\in\N}q^n a_n,\quad\text{where }a_n\in\mbh.\]
The converse is also true, a quaternion series of the form $\sum_{n\in\N}q^na_n$ also defines a slice regular function within its domain of convergence. By splitting the quaternion coefficients by complex numbers, the following lemma associates the slice regular functions with complex holomorphic functions.
\begin{lem}[Splitting Lemma]\label{lem:splitlemma}
    Let $f$ be a slice regular function on $B(0,R)$. Let $I,J\in\S$ such that $I\perp J$. Then there exist two complex holomorphic functions $F,G:B_{I}\to\C_{I}$ such that for any $z=x+Iy$ in $B_{I}$, we have
    \[f_{I}(z)=F(z)+G(z)J.\]
\end{lem}

Just as with holomorphic functions, an identity principle also holds for slice regular functions.

\begin{lem}[Identity principle]\label{lem:identityprcp}
Let $f$ and $g$ be two slice regular functions on $B(0,R)$. If there exists some $I\in\S$ such that the functions $f$ and $g$ coincide on a subset of $B_{I}$ with one accumulation point in $B_{I}$, then $f\equiv g$ holds on $B(0,R)$.
\end{lem}

The values of a slice regular function can be determined exactly by the values of its restrictions on a complex line. In particular, a holomorphic function $f_I:B_{I}\to \mbh$ admits a unique \emph{regular extension}  $\ext(f_{I})$ that has restriction $f_{I}$ on $B_I$.
\begin{lem}[Representation Formula]\label{lem:repformula}
    Let $f$ be a slice regular function on the ball $B=B(0,R)$ and let $I,J\in \mbs$. Then for all $x+yI\in B$ we have
    \[f(x+yJ)=\frac{1-JI}{2}f(x+yI)+\frac{1+JI}{2}f(x-yI).\]
Moreover, let $K\in\S$ and $I\neq J$, we have
\begin{align*}
    f(x+yK)&=(I-J)^{-1}[If(x+yI)-Jf(x+yJ)]\\
    &+K(I-J)^{-1}[f(x+yI)-f(x+yJ)].
\end{align*}

\end{lem}

Unfortunately, the product of two slice regular functions is usually no longer regular. A surrogate multiplication called the \emph{$\ast$-product}, which preserves slice regularity, is defined as below. 
\begin{defn}
    Let $f=\sum_{n\in\N}q^{n}a_n$ and $g=\sum_{n\in\N} q^{n}b_n$ be two slice regular functions on $B(0,R)$.
    The $\ast$-product of $f$ and $g$ is defined as the slice regular function $f\ast g$ on $B(0,R)$ with the power series
    \[(f\ast g)(q)=\sum_{n\in\N}q^{n}\sum_{l=0}^{n}a_{l}b_{n-l}.\]
\end{defn}
Note that the $\ast$-products are not commutative in general. Also, the $\ast$-product $(f\ast g)(q)$ is not equal to the product $f(q)g(q)$. However, we can still connect $\ast$-product and pointwise product by
\begin{lem}\label{lem:pointwiseprod}
    Let $f$ and $g$ be two slice regular functions on $B(0,R)$. Then we have
    \[(f\ast g)(p)=\begin{cases*}
        f(p)g(f(p)^{-1}pf(p))&\text{if }$f(p)\neq 0$,\\
        0&\text{if }$f(p)=0$.
    \end{cases*}\]
\end{lem}

We can associate a slice regular function $f$ with a particular power series $f^{s}$ with real coefficients, which vanishes at all zeros of $f$. In fact, if $f$ has one zero $a\in\mbh$, then $f^{s}$ vanishes on the $2$-sphere $\S_{a}$ (see~\cite[Theorem 4.3]{GSto1}).
\begin{defn}
   Let  $f=\sum_{n\in\N}q^{n}a_n$ be a slice regular function on $B(0,R)$. Then we define \emph{regular conjugate} of $f$ as the slice regular function
   \[f^{c}(q)\coloneqq\sum_{n\in\N}q^{n}\overline{a}_n,\]
   and the \emph{symmetrization} of $f$ as
   \[f^{s}(q)\coloneqq(f\ast f^{c})(q)=(f^{c}\ast f)(q)=\sum_{n\in\N}q^{n}\sum_{l=0}^{n}a_{l}\overline{a}_{n-l}.\]
\end{defn}

With the notations of regular conjugate and symmetrization, we give the following definition.
\begin{defn}\label{def:regrec}
    Let $f$ be a slice regular function on $B=B(0,R)$. Let the zero set of its symmetrization $\mathcal{Z}_{f^s}$ be a proper subset of $B$. The \emph{regular reciprocal} of $f$ is defined as the function $f^{-\ast}(q):B\backslash\mathcal{Z}_{f^s}\to\mbh$ by
    \[f^{-\ast}(q)\coloneqq (f^{s}(q))^{-1}f^{c}(q).\]
\end{defn}
If a slice regular function $f=\sum_{\ell=0}^n q^{\ell}a_{\ell}$ is given by a finite series, i.e. $a_{\ell}=0$ for $\ell\gg 1$, then $f$ is called a \emph{slice regular polynomial}. The following well-known theorem,  also recognized as \emph{Eilenberg--Niven theorem}, gives precise factorizations of slice regular polynomials.
\begin{thm}[Fundamental Theorem of Algebra for Quaternions~\cite{EN,GSto1,GP1}]\label{thm:FTAQ}
    Let $f(q)$ be a non-constant slice regular polynomial on $\mbh$. Then its zero set $\mathcal{Z}_{f}$ is non-empty. More precisely, the zero set $\mathcal{Z}_{f}$ is a finite union of spherical roots (of the form $\S_{\alpha}$ for some $\alpha\in\mbh$) and isolated roots. Also, the polynomial $f(q)$ can be factorized completely as
    \[f(q)=(q^2-2\RE \alpha_{1}\cdot q+\abs{\alpha_{1}}^2)^{m_1}\ast\dots(q^2-2\RE \alpha_{\ell}\cdot q+\abs{\alpha_{\ell}}^2)^{m_\ell}\ast\sideset{}{^*}\prod (q-\beta_h)c,\]
    for some $\alpha_{\ell}, \beta_{h},c\in\mbh$ and $\prod^*$ is the $*$-product.
\end{thm}

\begin{proof}
    See also \cite[Theorem 3.23]{GSS}.
\end{proof}

\begin{rmk}
 Note that the factorization of the slice regular polynomial $f(q)$ is usually not unique. Also, the quaternions $\beta_h\,(h>1)$ are in general not roots of $f(q)$. For example, the quadratic polynomial $f(q)=(q-i)\ast(q-j)$ has only one root $q=i$.
\end{rmk}

The following two kinds of functions are of particular importance in our computations.
\begin{defn}
  Let $f$ be a slice regular function on $B=B(0,R)$. Then $f$ is called \emph{slice preserving} if $f(B_I)\subset \mbc_I$ for any $I\in\mbs$. It is called \emph{one-slice preserving} on $\mbc_I$ if $f(B_I)\subset \mbc_I$ for one  certain $I\in\mbs$.
\end{defn}
It is possible to characterize these two kinds of functions by their power series.
\begin{thm}
    Let $f(q)=\sum_{n\in\mbn}q^na_n$ be a slice regular function on $B(0,R)$. Then we have 
    \begin{enumerate}
        \item The function $f$ is slice preserving precisely if $a_n\in \mbr$ for all $n$.
        \item The function $f$ is one-slice preserving on $\mbc_I$ precisely if $a_n\in \mbc_I$ for all $n$.
    \end{enumerate}
\end{thm}
\begin{proof}
See~\cite[Remark 1.42]{GSS}.
\end{proof}
Here we give an example of a slice regular function which is slice preserving.
\begin{exa}
 The quaternionic exponential function $\exp(q)$ (see for example~\cite[Example 1.7]{GSS}) is defined by the slice regular power series
\[e^{q}\coloneqq\sum_{n=0}^{\infty}\frac{q^n}{n!},\quad\text{for }q\in\mbh.\]
It has all real coefficients, and is thus slice preserving. In fact, when $I^2=-1$ is an imaginary unit, we have Euler's formula $e^{I\theta}=\cos\theta+I\sin\theta$ for all $\theta\in\R$. So it holds that
\[e^q=e^{\RE q}\Big(\!\cos |\!\IM q|+\frac{\IM q}{|\!\IM q|}\sin |\!\IM q|\Big).\]
The logarithm (principal branch) is its inverse
\[\log q=\log\abs{q}+\frac{\IM q}{\abs{\!\IM q}}\arccos{\frac{\RE q}{|q|}},\]
where $q\in \mbh\setminus \mbr_{\leq 0}$. One can find more detail in \cite[Section 4.2]{GSS}. In the $\SU(2)$-model of the unit quaternions, the quaternion exponential function is given by
\begin{align*}
  \exp:\mathfrak{su}(2)&\to \SU(2)\\
  \begin{pmatrix}
    i \theta_1&\theta_2+i \theta_3\\
    -\theta_2+i\theta_3&-i\theta_1
  \end{pmatrix}
  &\mapsto \begin{pmatrix}
    \cos \theta+i\frac{\sin \theta}{\theta}\theta_1&\frac{\sin\theta}{\theta}(\theta_2+i\theta_3)\\
    \frac{\sin\theta}{\theta}(-\theta_2+i\theta_3)&\cos \theta-i\frac{\sin \theta}{\theta}\theta_1
  \end{pmatrix},
\end{align*}
where $\theta=\sqrt{\theta_1^2+\theta_2^2+\theta_3^2}$. 
\end{exa}

\begin{lem}\label{lem:theta}
 Let $\mathbb{D}=\{i\theta_1+j\theta_2+k\theta_3\in\mbh\,|\,\theta=\sqrt{\theta_1^2+\theta_2^2+\theta_3^2}\in [0,\pi)\}$. Then the quaternion functions $\exp$ and $\log$ define a bijection between
\[
 \begin{tikzcd}[every arrow/.append style={shift left}]
\mathbb{D}\arrow{r}{\exp}&
\mbt^{1}(\mbh)\backslash\{-1\}\arrow{l}{\log}
 \end{tikzcd}
\]
\end{lem}
\begin{proof}
   Let $q\in\mathbb{D}$. Notice first that $\abs{e^{q}}=e^{\RE q}$.  Since we have $\theta\neq\pi$, it follows that $e^{q}\neq -1$ is in $\mbt^{1}(\mbh)$. For the other side, let $q\in\mbt^{1}(\mbh)$ and $q\neq -1$. We have $\RE \log q=0$ and $\arccos\frac{\RE q}{\abs{q}}\in [0,\pi)$, which indicates that $\log q\in\mathbb{D}$. This means that both maps are well-defined. A direct calculation shows that $\exp\circ\log$ and $\log\circ\exp$ are identity maps.
\end{proof}
We end this section by giving the formulas of Mahler measure in terms of quaternionic exponential.
\begin{thm}\label{thm:mahlerexp}
    Let $P$ be any quaternionic polynomial and let $\mathbb{D}$ be defined as before, then we have
    \begin{align*}
         &\mH\left(P(x_1,\dots,x_n)\right)\\
         &=\frac{1}{{(2\pi^2)}^n}\int_{{\mathbb{D}}^n}\log\Abs{P(e^{i\theta_{11}+j\theta_{12}+k\theta_{13}},\dots,e^{i\theta_{n1}+j\theta_{n2}+k\theta_{n3}})}\prod_{\ell=1}^{n}\frac{\sin^2 \theta_{\ell}}{\theta_\ell^2}\dif\theta_{\ell1}\dif\theta_{\ell2}\dif\theta_{\ell3},
    \end{align*}
and
    \begin{align*}
         &Z_{\mbh}\left(P(x_1,\dots,x_n),s\right)\\
         &=\frac{1}{{(2\pi^2)}^n}\int_{{\mathbb{D}}^n}\Abs{P(e^{i\theta_{11}+j\theta_{12}+k\theta_{13}},\dots,e^{i\theta_{n1}+j\theta_{n2}+k\theta_{n3}})}^{s}\prod_{\ell=1}^{n}\frac{\sin^2 \theta_{\ell}}{\theta_\ell^2}\dif\theta_{\ell1}\dif\theta_{\ell2}\dif\theta_{\ell3},
    \end{align*}
where we denote $\theta_\ell=\sqrt{\theta_{\ell1}^2+\theta_{\ell2}^2+\theta_{\ell3}^2}$.
\end{thm}
\begin{proof}
    In~\eqref{eq:SU(2)}, putting $z=\cos \theta+i\frac{\sin \theta}{\theta}\theta_1$ and $w=\frac{\sin\theta}{\theta}(\theta_2+i\theta_3)$, one finds that
    \[\frac{\dif z}{z}\wedge \dif w\wedge \dif \overline{w}=2\,\frac{\sin^2\theta}{\theta^2}\dif\theta_1\wedge\dif\theta_2\wedge\dif\theta_3.\]
Then using Lemma~\ref{lem:theta} we derive the desired formulas.
\end{proof}

\section{Jensen formula and Mahler measures of slice-preserving  polynomials}\label{sec4}

As a preparatory step, we give some formulas that we will be using to compute the quaternionic Mahler measures.

\begin{lem}[Jensen's formula]\label{lem:jenformula}
    Let $I\in\S$ and let $\alpha=a e^{I\phi}$ with $a\geq0$ and $\phi\in\R$. Define the function
    \[\mathcal{J}(\alpha)\coloneqq \mathcal{J}(a,\phi)=\frac{1}{\pi}\int_{0}^{\pi}\log\left(1-2a\cos(\theta-\phi)+a^2\right)\sin^2{\theta}\dif\theta.\]
    Then $\mathcal{J}(0)=0$, and for $0<\abs{\alpha}\leq 1$ we have
\begin{align*}
\mathcal{J}(\alpha)=\frac{\abs{\alpha}^2}{4}\cos2\phi-\frac{1}{2\pi I}\IM\biggl[&\Big(\alpha-\frac{1}{\alpha}\Big)-\Big(\alpha^2-\frac{1}{\alpha^2}\Big)\arctanh\alpha\\
&+4\Li_2(\alpha)-\Li_2(\alpha^2)\biggr].
\end{align*}
For $\abs{\alpha}>1$ we have
\begin{align*}
\mathcal{J}(\alpha)=\mathcal{J}(\alpha/\abs{\alpha}^2)+\log \abs{\alpha},
\end{align*}
where the functions $\Li_2$ and $\arctanh$ are regarded as the regular extensions of their complex counterparts.
\end{lem}

\begin{proof} 
The case $a=0$ is immediate. Hence assume $a>0$.
Note that we have
\[
    \mathcal{J}(a,\phi)-\log{\abs{a}}=\frac{1}{\pi}\int_0^{\pi}\log\Big(1-\frac{2}{a}\cos(\theta-\phi)+\frac{1}{a^{2}}\Big)\sin^2\theta\dif \theta.
\]
The latter is nothing but $\mathcal{J}(1/a,\phi)$, allowing us to assume that $\abs{a}\leq 1$.  Consider the following expansion (see \cite[Example 9.1]{WW})
\begin{equation}\label{eq:logalpha}
    \log(1-2a\cos(\theta-\phi)+a^2)=-2\sum_{n\geq 1}\frac{a^n}{n}\cos n(\theta-\phi).
\end{equation}
It remains to compute the integral of the form 
\begin{equation}\label{eq:cossin}
    A_n=\int_0^{\pi}\cos n(\theta-\phi)\sin^2\theta\dif\theta.
\end{equation}
Using
\[
\cos\bigl(n(\theta-\phi)\bigr)
=\RE\bigl(e^{-in\phi}e^{in\theta}\bigr),\quad
\sin^2\theta=\frac12-\frac14\bigl(e^{2i\theta}+e^{-2i\theta}\bigr),
\]
the integral~\eqref{eq:cossin} is the real part of 
\begin{align*}
e^{-in\phi}\left(\frac{1}{2}\int_0^\pi e^{in\theta}\dif \theta-\frac{1}{4}\int_0^{\pi}e^{i(n+2)\theta}\dif \theta-\frac{1}{4}\int_0^{\pi}e^{i(n-2)\theta}\dif\theta\right).
\end{align*}
Therefore,
\[
A_n=\begin{cases}
    -\frac{4}{(n^2-4)n}\sin n\phi&n \text{ odd},\\
    0&n \text{ even, } n\neq 2,\\
    -\frac{\pi}{4}\cos 2\phi&n=2.
\end{cases}\]
It then follows that
\begin{align*}
   &\mathcal{J}(a,\phi)=\frac{a^2}{4}\cos2\phi+\sum_{n\text{ odd},\,n\geq 1}\frac{8a^n}{n^2(n^2-4)\pi}\sin n\phi.
\end{align*}
Through partial fraction decomposition, we obtain
\[\frac{1}{n^2(n^2-4)}=\frac{1}{16}\left(\frac{1}{n-2}-\frac{1}{n+2}-\frac{4}{n^2}\right).\]
It is also worth noting that for $\lvert \alpha \rvert < 1$, one has
\begin{equation}\label{eq:arctanh}
    \arctanh \alpha=\sum_{n\text{ odd},\,n\geq 1}\frac{\alpha^n}{n}.
\end{equation}
So
\begin{align*}
\mathcal{J}(\alpha)=\frac{\abs{\alpha}^2}{4}\cos2\phi-\frac{1}{2\pi I}\IM\biggl[&\Big(\alpha-\frac{1}{\alpha}\Big)-\Big(\alpha^2-\frac{1}{\alpha^2}\Big)\arctanh\alpha\\
&+4\Li_2(\alpha)-\Li_2(\alpha^2)\biggr].\qedhere
\end{align*}
\end{proof}

In fact, the Mahler measures become much easier to compute if the polynomial is slice preserving or one-slice preserving.
\begin{prop}\label{prop:pxtheta}
  Let $P(x)=\sum_{\ell=0}^nx^{\ell}a_{\ell}$ be a slice regular polynomial. 
  \begin{enumerate}
      \item If $P$ is slice preserving, then we have
  \[m_{\mbh}(P)=\frac{2}{\pi}\int_0^{\pi}\log |P(e^{i \theta})|\sin^2\theta\dif \theta,\]
  \[Z_{\mbh}(P,s)=\frac{2}{\pi}\int_0^{\pi}|P(e^{i \theta})|^s\sin^2\theta\dif \theta.\]
  \item If $P$ is one-slice preserving on $\mbc_I$, and let $I\perp J$. Then 
  \[m_{\mbh}(P)=\frac{1}{\pi}\int_0^{\pi}\int_0^{\pi}\log |P(\cos \theta+Ie^{J\theta_1}\sin \theta)|\sin^2\theta\sin \theta_1\dif \theta\dif \theta_1,\]
  \[Z_{\mbh}(P,s)=\frac{1}{\pi}\int_0^{\pi}\int_0^{\pi}|P(\cos \theta+Ie^{J\theta_1}\sin \theta)|^s\sin^2\theta\sin \theta_1\dif \theta\dif \theta_1.\]
  \end{enumerate}
 \end{prop}

\begin{proof}
Let $\theta=\sqrt{\theta_1^2+\theta_2^2+\theta_3^2}$. If $P$ is slice preserving, then the coefficients of $P$ are real.  Since $e^{i \theta_1+j\theta_2+k\theta_3}$ is similar to $e^{i\theta}$, we have $|P(e^{i\theta_1+j\theta_2+k\theta_3})|=|P(e^{i\theta})|$. %Let again $\theta=\sqrt{\theta_1^2+\theta_2^2+\theta_3^2}$. 
Changing the variables with spherical coordinates  $\theta_1=\theta\cos \psi_1$, $\theta_2=\theta\sin \psi_1\cos\psi_2$ and $\theta_3=\theta\sin\psi_1\sin\psi_2$, one obtains from Theorem~\ref{thm:mahlerexp} that
  \begin{align*}
    m_{\mbh}(P)=&\frac{1}{2\pi^2}\int_{\theta_1^2+\theta_2^2+\theta_3^2\leq \pi^2}\log|P(e^{i\theta})|\frac{\sin^2\theta}{\theta^2}\dif \theta_1\dif\theta_2\dif\theta_3\\
    =&\frac{1}{2\pi^2}\int_{0}^{2\pi}\int_{0}^{\pi}\int_0^{\pi}\log|P(e^{i\theta})|\sin^2\theta\sin \psi_1\dif \theta\dif \psi_1\dif \psi_2\\
    =&\frac{2}{\pi}\int_0^{\pi}\log |P(e^{i \theta})|\sin^2\theta\dif \theta.
  \end{align*}
The proof of $Z_{\mbh}(P,s)$ is the same. This proves the first claim. 

For the second claim, let $P$ be one-slice preserving on $\mbc_I$, then the coefficients of $P$ are in $\mbc_{I}$. We may represent $x$ in the hyperspherical form
\[x=\cos\theta+\sin\theta\cos \theta_1 I+\sin \theta\sin\theta_1\cos \theta_2 J +\sin \theta\sin\theta_1\sin\theta_2 K,\]
where $I\perp J$ and $K=IJ$. Let $\phi=\theta_2/2-\pi/4$ and $q=e^{I\phi}\in \mbc_I$, then a direct calculation shows that $x$ is similar to $q\inv xq=\cos\theta+Ie^{J\theta_1}\sin\theta$. So for any $x\in \mbt^1(\mbh)$, we have $|P(x)|=|P(\cos\theta+Ie^{J\theta_1}\sin\theta)|$. Again changing the variable, we get
\begin{align*}
    m_{\mbh}(P)=&\frac{1}{2\pi^2}\int_{0}^{2\pi}\int_{0}^{\pi}\int_0^{\pi}\log\abs{P(x)}\sin^2\theta\sin\theta_1\dif \theta\dif \theta_1\dif \theta_2\\
    =&\frac{1}{\pi}\int_{0}^{\pi}\int_0^{\pi}\log |P(\cos\theta+Ie^{J\theta_1}\sin\theta)|\sin^2\theta\sin \theta_1\dif \theta\dif \theta_1.\qedhere
  \end{align*}
\end{proof}

We now come to the proof of the part (1) in Theorem~\ref{thm:Jensenformulas}.
\begin{proof}[Proof of (1) in Theorem~\ref{thm:Jensenformulas}]
Let $\alpha\in \mbh$ be any quaternion. Then there exists $q_0\in\mbh$ such that $\alpha$ is similar to $\alpha_0=q_0\inv \alpha q_0\in \mbc$. So 
\begin{align*}
    m_{\mbh}(x-\alpha)=&\int_{\mbt^1(\mbh)}\log|q_0\inv xq_0-\alpha_0|\dif \mu\\
    =&\int_{\mbt^1(\mbh)}\log| x-\alpha_0|\dif \mu=m_{\mbh}(x-\alpha_0).
\end{align*}
This means we may always assume that $\alpha\in\C$. For the next step, we set $\alpha=|\alpha|e^{i\theta}$. Writing the integral by the $\SU(2)$-model of unit quaternions, we get
    \begin{align*}
        m_{\mbh}(&x-\alpha)
        %=&
        %\frac{1}{4\pi^2}\int_{|z|^2+|w|^2=1}\log|z+jw-\alpha|\,\frac{\dif z}{z}\dif w\dif \overline{w}\\
        =\frac{1}{8\pi^2}\int_{|z|^2+|w|^2=1}\log\bigl(|z-|\alpha|e^{i\theta}|^2+|w|^2\bigr)\,\frac{\dif z}{z}\dif w\dif \overline{w}\\
        &=\frac{1}{8\pi^2}\int_{|z|^2+|w|^2=1}\log\bigl(|ze^{-i\theta}-|\alpha||^2+|w|^2\bigr)\,\frac{\dif z}{z}\dif w\dif \overline{w}=m_{\mbh}(x-|\alpha|).
    \end{align*}
Now we can restrict our attention only to the case where $\alpha\in \mbr$. In this scenario, the polynomial $x-\alpha$ is slice-preserving. From Proposition \ref{prop:pxtheta}, we find
\begin{align*}
    m_{\mbh}(x-\alpha)=&\frac{2}{\pi}\int_0^{\pi}\log|e^{i\theta}-\alpha|\sin^2\theta\dif \theta\\
    =&\frac{1}{\pi}\int_0^{\pi}\log|1-2\alpha\cos\theta+\alpha^2|\sin^2\theta\dif \theta.
\end{align*}
Thus, the formula follows directly from Jensen's formula \ref{lem:jenformula}.
\end{proof}
We can also evaluate the quaternionic zeta Mahler measure of $x-\alpha$ for $\alpha\in\mbh$. This offers a quaternion analogue of Akatsuka's results~\cite[Theorem 1, 2]{Aka}.
\begin{thm}\label{thm:zetameaofx}
    Let $\alpha\in \mbh$. The quaternionic zeta Mahler measure of $x-\alpha$ is given below.
    \begin{enumerate}
        \item If $\abs{\alpha}=1$ and $\RE s>-3$, we have
        \[
        Z_{\mbh}(x-\alpha;s)=\frac{2^{2+s}}{\sqrt{\pi}}\frac{\Gamma(\frac{s}{2}+\frac{3}{2})}{\Gamma(\frac{s}{2}+3)}.
        \]
        The right-hand side gives the meromorphic continuation of $Z_{\mbh}(x-\alpha;s)$ to all $s\in\C$.
        \item If $\abs{\alpha}<1$ and $s\in\C$, then 
        \[
        Z_{\mbh}(x-\alpha;s)=(1-|\alpha|)^s{}_{2}F_{1}\left(-\frac{s}{2},\frac{3}{2},3;-\frac{4|\alpha|}{(1-|\alpha|)^2}\right).
        \]
        
        \item If $\abs{\alpha}>1$, and $s\in\C$, then 
        \[
        Z_{\mbh}(x-\alpha;s)=(|\alpha|-1)^s{}_{2}F_{1}\left(-\frac{s}{2},\frac{3}{2},3;-\frac{4|\alpha|}{(1-|\alpha|)^2}\right).
        \]
        
    \end{enumerate}
\end{thm}

\begin{proof}
    Using the same reasoning as in the previous proof, we can safely assume that $\alpha\in \mbr^+$. Referring to Proposition \ref{prop:pxtheta}, we have
    \[Z_{\mbh}(x-\alpha;s)=\frac{2}{\pi}\int_{0}^{\pi}\abs{1-2\alpha\cos\theta+\alpha^2}^{\frac{s}{2}}\sin^2\theta\dif\theta.\]
    
    Let $\alpha=1$, then for $\RE s>-3$, one has
\[
Z_{\mbh}(x-\alpha;s)=\frac{2^{1+\frac{s}{2}}}{\pi}\int_{0}^{\pi}(1-\cos\theta)^{\frac{s}{2}}\sin^2\theta\dif\theta.
\]
Changing the variable by $u=\sin^2\frac{\theta}{2}$, then 
\[
1-\cos\theta=2u,\,\sin^2\theta=4u(1-u),\,\dif\theta=\dfrac{\dif u}{\sqrt{u(1-u)}}.
\]
The original integral becomes
\begin{align*}
     Z_{\mbh}(x-\alpha;s)
     &=\frac{2^{s+3}}{\pi}\int_0^1 u^{\frac{s+1}{2}}(1-u)^{\frac12}du\\
&=\frac{2^{s+2}}{\sqrt{\pi}}\frac{\Gamma\left(\frac{s}{2}+\frac32\right)}{\Gamma(\frac{s}{2}+3)}. 
\end{align*}
This quotient of Gamma functions is meromorphic in $s$ and gives the continuation of $Z_{\mbh}(x-\alpha;s)$ beyond the half-plane $\RE s>-3$. Its possible poles are simple and occur at $s=-3,-5,-7,\dots$.

If $0<\alpha<1$, changing the variable to $u=\sin^2\frac{\theta}{2}$ again, similarly we have 
\begin{align*}
    Z_{\mbh}(x-\alpha;s)&=\frac{8}{\pi}\int_0^1((1-\alpha)^2+4\alpha u)^{\frac{s}{2}}\sqrt{u(1-u)}\dif u\\
    &=\frac{8(1-\alpha)^s}{\pi}\int_0^1u^{\frac{1}{2}}(1-u)^{\frac{1}{2}}\Big(1+\frac{4\alpha}{(1-\alpha)^2}u\Big)^{\frac{s}{2}}\dif u\\
    &=(1-\alpha)^s{}_{2}F_{1}\left(-\frac{s}{2},\frac{3}{2},3;-\frac{4\alpha}{(1-\alpha)^2}\right).
\end{align*}
Here we use the integral representation~\cite[Example~14.6.1]{WW}, valid for $\RE(c)>\RE(b)>0$,
\[
\frac{\Gamma(b)\Gamma(c-b)}{\Gamma(c)}\,{}_2F_1(a,b;c;z)
=\int_{0}^{1} x^{b-1}(1-x)^{c-b-1}(1-zx)^{-a}\,dx.
\]

If $\alpha>1$, we have the functional equation $Z_{\mbh}(x-\alpha^{-1};s)=\alpha^{-s}Z_{\mbh}(x-\alpha;s)$. This reduces to the above case.
\end{proof}

We arrive at proof of the part (2) in Theorem~\ref{thm:Jensenformulas}.
\begin{proof}[Proof of (2) in Theorem~\ref{thm:Jensenformulas}]
    Again, since every quaternion is similar to a complex number, without loss of generality we may assume that $\alpha=ae^{i\phi}\in \mbc$. By Proposition~\ref{prop:pxtheta}, we have 
    \begin{align*}
        m_{\mbh}(x^2-2&\RE \alpha\cdot x+|\alpha|^2)=\frac{2}{\pi}\int_0^{\pi}\log|e^{2i\theta}-2\RE \alpha\cdot e^{i\theta}+|\alpha|^2|\sin^2\theta\dif \theta\\
        &=\frac{2}{\pi}\int_0^{\pi}(\log|e^{i\theta}-ae^{i\phi}|+\log|e^{i\theta}-ae^{-i\phi}|)\sin^2\theta\dif \theta.
    \end{align*}
    It is evident from Lemma \ref{lem:jenformula} that the last integral is $\mathcal{J}(\alpha)+\mathcal{J}(\overline{\alpha})$.
\end{proof}

To give the final proof of the part (3) in Theorem~\ref{thm:Jensenformulas}, we start with the following theorem.

\begin{thm}\label{thm:pq}
  Let $P(x)\in\mbr[x]$ be a slice-preserving polynomial and $Q(x)$ be any quaternion slice regular polynomial. Then we have
  \begin{align*}
    m_{\mbh}&(P\ast Q)=m_{\mbh}(Q\ast P)\\
  &=m_{\mbh}(PQ)=m_{\mbh}(QP)=m_{\mbh}(P)+m_{\mbh}(Q). 
  \end{align*}
\end{thm}

\begin{proof}
    Since $P$ is slice preserving and having real coefficients, we know everything commutes. It turns out that $P\ast Q=Q\ast P$ (the $\ast$-product) is exactly the $PQ=QP$ (the non-commutative product).
\end{proof}

\begin{rmk}
We point out that $m_{\mbh}(P\ast Q)=m_{\mbh}(P)+m_{\mbh}(Q)$ is usually false if $P,Q$ are not slice-preserving polynomials. For instance, for $P(x)=x-i$ and $Q(x)=x+i$ it follows
\[m_{\mbh}((x-i)\ast (x+i))=m_{\mbh}(x^2+1)=-\frac{1}{2},\]
while
\[m_{\mbh}(x-i)+m_{\mbh}(x+i)=\frac{1}{2}.\]
We will see in Section \ref{sec:mahmeaofleftpol} it still fails even when $P=Q$. A counterexample is 
\[m_{\mbh}((x-i)*(x-i))=\frac{5}{6}\neq 2m_{\mbh}(x-i).\]
\end{rmk}

\begin{proof}[Proof of (3) in Theorem~\ref{thm:Jensenformulas}]
This is an immediate corollary of  Theorem~\ref{thm:pq}.
\end{proof}

\section{Mahler measures of slice regular polynomials}\label{sec:mahmeaofleftpol}
In this section, we will focus on Mahler measures of slice regular polynomials. 
We commence by establishing the existence of quaternionic Mahler measures. To validate the proof, we introduce an important lemma which is essential for demonstrating the existence of the complex Mahler measure.
\begin{lem}[Lemma 3.8 in Everest--Ward~\cite{EW}]~\label{lem:EW}
    Let $P\in\C[x_1^{\pm},\dots,x_n^{\pm}]$ be a complex Laurent polynomial. Then there exists $C,\delta>0$ depending only on $P$ such that for any $\varepsilon>0$, one has that 
    \[\mu_{\C}\left(\big\{(x_1,\dots,x_n)\in\mbt^{n}(\C)\,\big{|}\,\abs{P(x_1,\dots,x_n)}\leq\varepsilon\big\}\right)\leq C\varepsilon^{\delta},\]
where $\mu_{\C}$ represents the probability Haar measure on $\mbt^{n}(\C)$.
\end{lem}

The following lemma introduces a quaternionic analogue for analysing the singularities in one single variable, similar to its complex counterpart shown previously.
\begin{lem}\label{lem:measurebound}
    Let $P(x)$ be a non-zero slice regular quaternionic polynomial. Let $\mu_{\mbh}$ be the probability Haar measure of the unit quaternion. Then  there exists $C,\delta>0$ depending only on $P$ such that for any $\varepsilon>0$, one has
    \[\mu_{\mbh}\left(\big\{q\in\mbt^{1}(\mbh)\,\big{|}\,\abs{P(q)}\leq\varepsilon\big\}\right)\leq C\varepsilon^{\delta}.\]
\end{lem}
\begin{proof}
    Let us firstly prove the assertion when $P$ is a slice-preserving polynomial. Let $q\in\mbt^{1}(\mbh)$ be with $\abs{P(q)}\leq\varepsilon$. Then for any $p\in\S_{q}$, we still have $\abs{P(p)}\leq\varepsilon$. Thus
    \[
    \mu_{\mbh}\left(\big\{q\in\mbt^{1}(\mbh)\,\big{|}\,\abs{P(q)}\leq\varepsilon\big\}\right)\leq 2\mu_{\C}\left(\big\{q\in\mbt^{1}(\C)\,\big{|}\,\abs{P(q)}\leq\varepsilon\big\}\right),
    \]
where $\mu_{\C}$ is probability Haar measure of the unit complex torus $\mbt^{1}(\C)$. By Lemma~\ref{lem:EW}, we have the bound
\[\mu_{\C}\left(\big\{q\in\mbt^{1}(\C)\,\big{|}\,\abs{P(q)}\leq\varepsilon\big\}\right)\leq C\varepsilon^{\delta}.\]

Secondly, we prove the assertion when $P$ is an arbitrary slice regular polynomial. Let $I,J\in\S$ be two orthogonal imaginary units. By the Splitting Lemma~\ref{lem:splitlemma}, there are two polynomials $F,G:\C_{I}\to\C_{I}$ such that  
\[P(q)=F(q)+G(q)J\quad\text{and}\quad\abs{P(q)}^2=\abs{F(q)}^2+\abs{G(q)}^2,\]
Meanwhile, its symmetrization is a slice-preserving polynomial with
    \[P^{s}(q)=F(q)\overline{F(\overline{q})}+G(q)\overline{G(\overline{q})}\quad\text{and}\quad\abs{P^{s}(q)}^2=\abs{F(q)\overline{F(\overline{q})}+G(q)\overline{G(\overline{q})}}^2. \]
Via the Cauchy--Schwarz inequality, one finds that $\abs{P^{s}(q)}\leq\abs{P(q)}\abs{P(\overline{q})}$. So that if $\abs{P(q)}\leq \varepsilon$, one has 
\[\abs{P^{s}(q)}\leq M\varepsilon,\quad\text{where } M=\max_{q\in\mbt^{1}(\mbh)}\abs{P(q)}.\]
However the measure of the latter is controlled by $\leq C\varepsilon^{\delta}$. Thus we know
\[\mu_{\mbh}\left(\big\{q\in\mbt^{1}(\mbh)\,\big{|}\,\abs{P(q)}\leq\varepsilon\big\}\right)\leq C'\varepsilon^{\delta}.\qedhere
\]
\end{proof}

%The intersection of closed set
%$$C=\bigcap_{N\in\N}\big\{q\in\mbt^{1}(\mbh)\,\big{|}\,\abs{P(q)},\,\abs{P(\overline{q})}\leq\frac{1}{N}\big\}$$
%consists of all zeros $q$ such that $P(q)=P(\overline{q})=0$. However, from the Representation Formula~\ref{lem:repformula} we know that $P$ must have spherical zeros on $\S_{q}$. Thus the set $C$ has to be empty. Hence, we know there exists $N>0$ such that whenever $\abs{P(q)}<\frac{1}{N}$, one has always $\abs{P(\overline{q})}>\frac{1}{N}$.
At this stage, we are now prepared to demonstrate Theorem~\ref{thm:existence}.
\begin{proof}[Proof of Theorem~\ref{thm:existence}]
From Lemma~\ref{lem:measurebound} we see 
\[\int_{\abs{P}\leq\varepsilon}\log\abs{P}\dif\mu(\mbt^1(\mbh))=O(\varepsilon^{\delta}\log\varepsilon),\]
as $\varepsilon\to 0$. Therefore, the integral converges absolutely as an improper Riemann integral.

Next we prove the asserted inequality. Let $z=\cos\theta+I\sin\theta$, then again by the Splitting Lemma~\ref{lem:splitlemma}, there are two polynomials $F,G:\C_{I}\to\C_{I}$ such that  
\[P(z)=F(z)+G(z)J\quad\text{and}\quad\abs{P(z)}^2=\abs{F(z)}^2+\abs{G(z)}^2,\]
and then
    \begin{align*}      
    \mH(P)&=\frac{1}{\pi}\int_{0}^{2\pi}\int_{\S}\log\abs{P(\cos\theta+I\sin\theta)}\sin^2\theta\dif\mu(I)\dif\theta\\
    &=\frac{1}{2\pi}\int_{0}^{2\pi}\int_{\S}\log(\abs{F(z)}^2+\abs{G(z)}^2)\sin^2\theta\dif\mu(I)\dif\theta.
    \end{align*}
Using Cauchy--Schwarz once more, we find
\[
 2\mH(P)\geq \frac{1}{2\pi}\int_{0}^{2\pi}\int_{\S}\log\abs{F(z)\overline{F(\overline{z})}+G(z)\overline{G(\overline{z})}}^2\sin^2\theta\dif\mu(I)\dif\theta=\mH(P^{s}),
\]
where the equality holds if and only if there exists some fixed $\lambda\in \mbc_I$ such that $\overline{F(\overline{z})}=\lambda F(z)$ and $\overline{G(\overline{z})}=\overline{\lambda}G(z)$ for any $z\in\mbh$ (using the Identity Principle~\ref{lem:identityprcp}). As $I,J$ vary over $\S$, we know that $\lambda=1$ and $\overline{G(\overline{z})}=0$. This is exactly $P(z)=\overline{P(\overline{z})}$, that is, $P$ is slice preserving.
\end{proof}

In light of Theorem~\ref{thm:existence}, we now turn to the task of determining the Mahler measure of slice regular polynomials. Using the Fundamental Theorem of Algebra for Quaternions~\ref{thm:FTAQ}, we arrive at the following problem.
\begin{problem}
   Given a slice regular quaternionic polynomial $P(x)$ with its complete factorization
  \[(q^2-2\RE \alpha_{1}\cdot q+\abs{\alpha_{1}}^2)^{m_1}\ast\dots(q^2-2\RE \alpha_{\ell}\cdot q+\abs{\alpha_{\ell}}^2)^{m_\ell}\ast\sideset{}{^*}\prod (q-\beta_h)c.\] 
Determine its Mahler measure $\mH(P)$.
\end{problem}
In view of Theorem~\ref{thm:pq}, our objective is narrowed down to find the Mahler measure of the $\ast$-product of monomials. In contrast to real polynomials, there are no straightforward solutions, since $m_{\mbh}(P\ast Q)=m_{\mbh}(P)+m_{\mbh}(Q)$ does not hold in general.

While the general task can indeed be quite challenging, there are specific cases where the integrals involved become simpler to deal with. For example, when $P$ is one-slice preserving, the computation is more approachable. The following proof provides us with an example illustrating how to obtain the Mahler measures of $(x-i)^{*n}$.

\begin{proof}[Proof of Theorem~\ref{thm:iteratedi}]
    Let $q=a+rJ$ with $J=i \cos\theta+j\sin\theta\cos\psi+k\sin\theta\sin\psi$. By the Representation Formula~\ref{lem:repformula}, we get 
    \begin{equation}\label{eq:P*n}
        P^{*n}(q)=\frac{1-Ji}{2}P(a+ri)^n+\frac{1+Ji}{2}P(a-ri)^n.
    \end{equation}
    Let $u=P(a+ri)^n$ and $v=P(a-ri)^n$ denote the two complex polynomials in $a$ and $r$. Then
    \begin{align*}
        \Bigl|\frac{1-Ji}{2}u&+\frac{1+Ji}{2}v\Bigr|^{2}=\Bigl|\frac{1+\cos \theta}{2}u+\frac{1-\cos\theta}{2}v+\frac{w(\bar{u}-\bar{v})}{2}k\Bigr|^2\\
        &=\frac{(1+\cos\theta)^2}{4}|u|^2+\frac{(1-\cos\theta)^2}{4}|v|^2\\
        &\hspace*{4em}+2\RE\frac{1-\cos^2\theta}{4}u\bar{v}+|w|^2\frac{|u|^2+|v|^2-2\RE u\bar{v}}{4},
    \end{align*}
    where $w=\sin \theta\cos\psi+i\sin\theta\sin\psi$. After simplification, we obtain
    \[
     \Bigl|\frac{1-Ji}{2}u+\frac{1+Ji}{2}v\Bigr|^{2}=\frac{1+\cos\theta}{2}|u|^{2}+\frac{1-\cos\theta}{2}|v|^{2}.
    \]
    Besides, 
    \begin{gather*}
    |u|^2=|a+(r-1)i|^{2n}=(a^2+(r-1)^2)^n=2^n(1-r)^n,\\
    |v|^2=|a-(r+1)i|^{2n}=(a^2+(r+1)^2)^n=2^n(1+r)^n,
    \end{gather*}
    using $a^2=1-r^2$. Plugging into Eq.~\eqref{eq:P*n}, we get
    \[
        |P^{*n}(a+rJ)|^2=2^n\left(A^+(r)-A^-(r)\cos\theta\right),
    \]
    where
    \[
    A^+(r)=\sum_{\ell \text{ even}}\binom{n}{\ell}r^{\ell},\quad A^-(r)=\sum_{\ell \text{ odd}}\binom{n}{\ell}r^{\ell}.
    \]
    
    Then by Proposition~\ref{prop:pxtheta}, we have  
    \begin{align*}
        &\mH(P^{*n})=\frac{1}{\pi}\int_0^1\int_0^{\pi}\log|P^{*n}(a+rJ)|^2\frac{r^2\sin\theta}{\sqrt{1-r^2}}\dif \theta\dif r\\
        =&\frac{n}{2}\log 2+\frac{1}{\pi}\int_0^1\int_0^{\pi}\log\left(A^+(r)-A^-(r)\cos\theta\right)\frac{r^2\sin\theta}{\sqrt{1-r^2}}\dif \theta\dif r.
    \end{align*}  
Note that for any real numbers $a>b>0$, we have 
\begin{equation}\label{eq:logabtheta}
    \int_0^{\pi}\log (a-b\cos\theta)\sin \theta\dif \theta
    =\log(a^2-b^2)+\frac{a}{b}\log\Big(\frac{a+b}{a-b}\Big)-2.
\end{equation}
 It can be deduced that 
    \begin{align*}
        \mH(P^{*n})=\frac{n}{2}\log 2+\frac{1}{\pi}&\int_0^1\big(\log(A^+(r)^2-A^-(r)^2)-2\big)\frac{r^2}{\sqrt{1-r^2}}\dif r\\
    &+\frac{1}{\pi}\int_0^1\left(\frac{A^+(r)}{A^-(r)}\log\frac{A^+(r)+A^-(r)}{A^+(r)-A^-(r)}\right)\frac{r^2}{\sqrt{1-r^2}}\dif r.
    \end{align*}
Changing variable with $r=\sin\theta$ in the first integral, we derive that \[A^+(r)^2-A^-(r)^2=\cos^{2n}\theta.\]
So we know the first line turns out to be $-\frac{1}{4}(n+2)$. In the second integral, substituting $r=\tanh t$ we find
\[\frac{A^+(r)}{A^-(r)}=\coth{nt},\quad\log\frac{A^+(r)+A^-(r)}{A^+(r)-A^-(r)}=2nt.\]
These yield the expected formula.
\end{proof}

Theorem~\ref{thm:iteratedi} can be extended to all integer powers by using regular reciprocals.
\begin{prop}\label{prop:reciprocali} Let $n\in\mathbb{Z}$. We set $(x-i)^{*0}=1$. Then
\[
m_{\mbh}((x-i)^{*n})=
\begin{cases}
0&\text{if }n=0,\\
\dfrac{2n}{\pi}\displaystyle\int_0^{\infty}t\coth nt\tanh^2t\sech t\dif t-\dfrac{2+n}{4}&\text{if }n\neq0.
\end{cases}
\]
\end{prop}
\begin{proof}
The case $n=0$ is immediate, and the case $n>0$ is Theorem~\ref{thm:iteratedi}. Let $n<0$ and let $(x-i)^{*n}$ denote the corresponding regular reciprocal power. Let $P(x)=x-i$. Its conjugate and symmetrization are
\[P^{c}(x)=x+i,\quad P^{s}(x)=x^2+1.\]
Thus
\[P^{*n}(x)=(x^2+1)^n(x+i)^{*(-n)}.\]
Therefore, using $\mH(x^2+1)=-1/2$, we obtain
\begin{align*}
    \mH(P^{*n})&=\mH((x+i)^{*(-n)})+n\mH(x^2+1)\\
    &=-\frac{2n}{\pi}\int_0^{\infty}t\coth(-nt)\tanh^2t\sech t\dif t-\frac{2+n}{4}\\
    &=\frac{2n}{\pi}\int_0^{\infty}t\coth nt\tanh^2t\sech t\dif t-\frac{2+n}{4}.\qedhere
\end{align*}
\end{proof}

\begin{cor}
    It can be seen that
    \begin{equation}\label{eq:mhn}
        \lim_{n\to+\infty}\frac{m_{\mbh}((x-i)^{*n})}{n}=\frac{1+2G}{\pi}-\frac{1}{4},\quad
        \lim_{n\to-\infty}\frac{m_{\mbh}((x-i)^{*n})}{n}=-\frac{1+2G}{\pi}-\frac{1}{4},
    \end{equation}
    where $G=L(\chi_{-4},2)$ is Catalan's constant.
\end{cor}

\begin{proof}
    By Proposition~\ref{prop:reciprocali}, for $n\neq0$ we have
\[
\frac{m_{\mbh}((x-i)^{*n})}{n}
=\frac{2}{\pi}\int_0^\infty t\coth nt\tanh^2t\sech t\dif t-\frac{2+n}{4n}.
\]
Notice that $\coth nt\to 1$ as $n\to+\infty$, while $\coth nt\to -1$ as $n\to-\infty$, and 
\[\int_0^{\infty}t\tanh^2t\sech t\dif t=\frac{1}{2}+G.\]
The two limits now follow by dominated convergence.
\end{proof}

Note that if we set $u=e^t$, then
\[\coth nt\tanh^2t\sech t=2u\frac{(u^{2n}+1)(u^2-1)^2}{(u^{2n}-1)(u^2+1)^3}.\]
By partial fractional decomposition we can always interpret $\mH((x-i)^{*n})$ with polylogarithms. In particular,
\[\mH(x-i)=\frac{1}{4},\quad \mH((x-i)^{*2})=\frac{5}{6},\]
\begin{align*}
  \mH((x-i)^{*3})&=\frac{31}{4}-\frac{12}{\pi}\big(\!\Li_2(\zeta_{6})-\Li_2(-\zeta_{6})-\Li_2(\zeta_{6}^2)+\Li_2(-\zeta_{6}^2)\big)\\
  &=\frac{31}{4}-2\pi,  
\end{align*}
where $\zeta_{k}=e^{\frac{2\pi i}{k}}$. And
\begin{align*}
\mH((x-i)^{*4})&=\frac{13}{3}-\frac{2\sqrt{2}}{\pi}\big(\!\Li_2(\zeta_{8})-\Li_2(-\zeta_{8})-\Li_2(\zeta_{8}^3)+\Li_2(-\zeta_{8}^3)\big)\\
&=\frac{13}{3}-\frac{\sqrt{2}}{2}\pi.
\end{align*}

Although no simple closed-form expression is expected for $\mH((x-i)^{*n})$, computations for small $n$ suggest the following result.
\begin{prop}
For any integer $n$, we have
    \[
     m_{\mbh}((x-i)^{*n}) \in \bar{\mathbb{Q}}+\bar{\mathbb{Q}}\pi
    \]
\end{prop}

\begin{proof}
   The case $n=0$ is immediate. Proposition~\ref{prop:reciprocali} implies that it suffices to treat the case $n>0$, as $\mH(x^2+1)=-1/2$. We first assume that $n$ is positive and odd. Changing the variable $u=e^t$ as above, by Theorem \ref{thm:iteratedi} we have
    \[m_{\mbh}((x-i)^{*n})=\frac{4n}{\pi}\int_1^{\infty}\frac{(u^{2n} + 1)(u^2 - 1)^2}{(u^{2n} - 1)(u^2 + 1)^3}\log u\dif u-\frac{2+n}{4}.\]
    By partial fractional decomposition, we have 
        \[\frac{(u^{2n} + 1)(u^2 - 1)^2}{(u^{2n} - 1)(u^2 + 1)^3}=\sum_{k=1}^{2n}\frac{a_k}{u-\zeta_{2n}^k}+\frac{b_1}{u-i}+\frac{b_2}{(u-i)^2}+\frac{c_1}{u+i}+\frac{c_2}{(u+i)^2},\]
        for some $a_k, b_1,b_2,c_1,c_2\in \bar{\mbq}$. Moreover, a direct computation shows that $a_k,b_2,c_2$ are real algebraic numbers with $a_k=a_{2n-k}=-a_{n-k}=-a_{n+k},b_1=c_1=0$ and $b_2=c_2=n/2$. Here we display the computation for $a_k\in \mbr$, the others are similar. In fact, $a_k$ can be expressed as 
        \[a_k=\Res_{u=\zeta_{2n}^k}\frac{(u^{2n} + 1)(u^2 - 1)^2}{(u^{2n} - 1)(u^2 + 1)^3}=\frac{2(\zeta_{2n}^{2k} - 1)^2}{\prod_{j\neq k}(\zeta_{2n}^j-\zeta_{2n}^k)(\zeta_{2n}^{2k} + 1)^3}.\]
        Taking conjugation, we get
        \begin{align*}
            \bar{a}_k=&\frac{2(\zeta_{2n}^{-2k} - 1)^2}{\prod_{j\neq k}(\zeta_{2n}^{-j}-\zeta_{2n}^{-k})(\zeta_{2n}^{-2k} + 1)^3}=\frac{2\zeta_{2n}^{2k}\prod_{j\neq k}(\zeta_{2n}^{j+k})(\zeta_{2n}^{2k}-1)^2}{\prod_{j\neq k}(\zeta_{2n}^k-\zeta_{2n}^j)(\zeta_{2n}^{2k}+1)^3}\\
            =&\frac{\zeta_{2n}^{2k}\zeta_{2n}^{k(2n-2)}\zeta_{2n}^{\frac{2n(2n+1)}{2}}}{(-1)^{2n-1}}\frac{2(\zeta_{2n}^{2k} - 1)^2}{\prod_{j\neq k}(\zeta_{2n}^j-\zeta_{2n}^k)(\zeta_{2n}^{2k} + 1)^3}=a_k.
        \end{align*}

        So we get 
        \begin{align*}
            \frac{(u^{2n} + 1)(u^2 - 1)^2}{(u^{2n} - 1)(u^2 + 1)^3}&=\sum_{k=1}^{(n-1)/2}a_k\Bigl(\frac{1}{u-\zeta_{2n}^k}-\frac{1}{u+\zeta_{2n}^k}+\frac{1}{u-\zeta_{2n}^{-k}}-\frac{1}{u+\zeta_{2n}^{-k}}\Bigr)\\
            &\qquad+b_2\Bigl(\frac{1}{(u-i)^2}+\frac{1}{(u+i)^2}\Bigr).
        \end{align*}
        By expanding directly, we get 
        \begin{align*}
            &\int_1^\infty\left(\frac{1}{u-\zeta_{2n}^k}-\frac{1}{u+\zeta_{2n}^k}+\frac{1}{u-\zeta_{2n}^{-k}}-\frac{1}{u+\zeta_{2n}^{-k}}\right)\log u\dif u\\
             &\quad=\sum_{m\ge 0}\left(\zeta_{2n}^{km}-(-1)^m\zeta_{2n}^{km}+\zeta_{2n}^{-km}-(-1)^m\zeta_{2n}^{-km}\right)\int_1^\infty \frac{\log u}{u^{2m+2}}\dif u\\
            &\quad=\Li_2(\zeta_{2n}^k)-\Li_2(-\zeta_{2n}^k)+\Li_2(\zeta_{2n}^{-k})-\Li_2(-\zeta_{2n}^{-k}).
        \end{align*}
        Meanwhile,
        \[
        \begin{aligned}
        &\int_1^{\infty}\Bigl(\frac{1}{(u-\mmi)^2}+\frac{1}{(u+\mmi)^2}\Bigr)\log u\dif u
        =\int_1^\infty \frac{2(u^2-1)}{(u^2+1)^2}\log u\dif u\\
        &\quad=-2\int_1^\infty \frac{\dif}{\dif u}\Bigl(\frac{u}{u^2+1}\Bigr)\log u\dif u=2\int_1^\infty \frac{1}{u^2+1}\dif u=\frac{\pi}{2},
        \end{aligned}
        \]
        where in the third equality we integrated by parts.
        % \[\int_1^{\infty}\Bigl(\frac{1}{(u-i)^2}+\frac{1}{(u+i)^2}\Bigr)\log u\dif u=\frac{\pi}{2}.\]
        Moreover, we have~\cite[25.12.8]{NIST} 
        \[\Li_2(e^{i\theta})=i\operatorname{Cl}_2(\theta)+\frac{\theta^2-2\pi\theta}{4}+\frac{\pi^2}{6},\]
        where $\operatorname{Cl}_2$ denotes the Clausen function. So
        \[\Li_2(\zeta_{2n}^k)-\Li_2(-\zeta_{2n}^k)+\Li_2(\zeta_{2n}^{-k})-\Li_2(-\zeta_{2n}^{-k})=\frac{\pi^2}{2}-\frac{k\pi^2}{n}.\]
        Combining the computation above, we get
        \[
        m_{\mbh}((x-i)^{*n})\in \bar{\mbq}+\bar{\mbq}\pi.
        \]

        The case of even $n$ is analogous. The only difference is that the function
        \[
        \frac{(u^{2n} + 1)(u^2 - 1)^2}{(u^{2n} - 1)(u^2 + 1)^3}
        \]
        has poles of order $4$ at $u=\pm i$, leading to higher order terms in the partial fraction decomposition, but the calculations remain unchanged.
\end{proof}

In order to generalize the above formula, we first introduce the \emph{$\sigma$-distance} which appears in Gentili--Stoppato~\cite{GSto2}.
\begin{defn}
    Let $p,q\in \mbh$, the $\sigma$-distance between them is given by 
    \[\sigma(q,p)=\begin{dcases}
        |q-p|&\text{if $p,q\in \mbc_I$ for some $I\in \mbs$},\\
        \sqrt{(\RE q-\RE p)^2+(|\!\IM q|+|\!\IM p|)^2}& \text{otherwise}.
    \end{dcases}\]
In particular, if $p\in\C_{I}$, then
\[\sigma(q,p)=\max_{z\in\C_{I}\cap\S_{q}}\left\{\abs{z-p},\,\abs{\overline{z}-p}\right\}.\]
\end{defn}
The following lemma is a strengthened version of the inequality presented in Gentili--Stoppato~\cite[Theorem 6]{GSto2}.
\begin{lem}\label{lem:sigmaqp}
    Let $I,J\in\mbs$, $p\in \mbc_I$ and $n\geq1$. Then for any $q\in \mbc_J$, we have 
    \[\frac{1}{2n}\log\Big(\frac{1}{2}(1-|\langle I,J\rangle|)\Big)\leq \frac{1}{n}\log|(q-p)^{*n}|-\log\sigma(q,p)\leq 0,\]
    where $\langle I,J\rangle$ is the usual Euclidean inner product on $\mathbb{R}^{3}$.
\end{lem}
\begin{proof}
    For convenience, we can put $p\in\C$ and $I=i$. Let $J=i\cos\theta +j\sin\theta\cos\theta_{1}+k\sin\theta\sin\theta_{1}$.
    If $q\in\C$, the inequalities become obvious because $|(q-p)^{*n}|^{\frac{1}{n}}=|q-p|=\sigma(q,p)$.
If $q\in\C_{J}$ with $J\neq\pm i$, then by Representation Formula~\ref{lem:repformula} for $z\in\C_{I}\cap\S_{q}$ we have
\[(q-p)^{*n}=\frac{1-JI}{2}(z-p)^n+\frac{1+JI}{2}(\overline{z}-p)^n.\]
If $\abs{z-p}\leq\abs{\overline{z}-p}$ then the distance $\sigma(p,q)=\abs{\overline{z}-p}$. Hence, we obtain
\begin{align*}
  \frac{1}{n}\log|(q-p)^{*n}|&-\log\sigma(q,p)=\frac{1}{n} \log\Big|\frac{1-JI}{2}\Big(\frac{z-p}{\overline{z}-p}\Big)^n+\frac{1+JI}{2}\Big|\\
  &=\frac{1}{2n}\log\Big(\frac{1}{2}(1-\cos\theta)+(1+\cos\theta)\Big|\frac{z-p}{\overline{z}-p}\Big|^{n}\Big).
\end{align*}
One finds thusly
\[\frac{1}{2n}\log\Big(\frac{1}{2}(1-\cos\theta)\Big)\leq\frac{1}{n}\log|(q-p)^{*n}|-\log\sigma(q,p)\leq 0.\]
The case $\abs{z-p}>\abs{\overline{z}-p}$ can be treated in the same manner.
\end{proof}

\begin{thm}\label{thm:iterateda}
    Let $\alpha\in \mbh$ with $\alpha\ne0$. Let $I$ be such that $\alpha=\RE \alpha-I\abs{\!\IM \alpha}$, and put $\phi=-\arccos \frac{\RE \alpha}{|\alpha|}\in[-\pi,0]$. Then
    \[
    \lim_{n\to +\infty}\frac{m_{\mbh}((x-\alpha)^{*n})}{n}=\mathcal{J}(\alpha),
    \]
    where $\mathcal{J}$ is the function in Lemma~\ref{lem:jenformula}.
\end{thm}

\begin{proof}[Proof of Theorem~\ref{thm:iterateda}]
    Let $P(x)=x-\alpha$. By similarity, we may assume that $\alpha\in \mbc$ and $I=i$. By Lemma \ref{lem:sigmaqp}, we have
    \begin{align*}
        \biggl|\frac{m_{\mbh}(P^{*n})}{n}&-\int_{\mbt^1(\mbh)}\log \sigma(q,\alpha)\dif \mu(q)\biggr| \\
        &\leq\frac{1}{\pi n}\left|\int_0^{\pi}\int_{\mbs}\log\Big(\frac{1-|\langle i,J\rangle|}{2}\Big)\sin^2\theta\dif J\dif \theta\right|.
    \end{align*}
    Here we have rewritten the Haar measure on $\mbt^1(\mbh)\cong S^3$ in polar coordinates
    \[
    q=\cos\theta+\sin\theta J,\qquad \theta\in[0,\pi],\ J\in\mbs,
    \]
    so that
    \[
    \dif\mu(q)=\frac{2}{\pi}\sin^2\theta\dif J\dif\theta,
    \]
where $\dif J$ denotes the Haar measure of $\mbs$. The singularity set $\{J\in\mbs\mid|\langle i,J\rangle|=1\}=\{\pm i\}$ is finite and the integrand has at most a logarithmic singularity there, so the integral on the right-hand side is finite.
    
    This implies that 
    \[\lim_{n\to \infty}\frac{m_{\mbh}(P^{*n})}{n}=\int_{\mbt^1(\mbh)}\log \sigma(q,\alpha)\dif \mu(q).\]
    For any $q\in \mbh\setminus \mbc$ and $p\in \mbh$, we have $\sigma(q,\alpha)=\sigma(p\inv qp,\alpha)$. Therefore, using the same method as in Proposition~\ref{prop:pxtheta}, it can be deduced that
    \begin{align*}
        \int_{\mbt^1(\mbh)}\log \sigma(q,\alpha)\dif \mu(q)=&\frac{1}{\pi}\int_0^{\pi}\log\bigl((\cos\theta-\RE \alpha)^2+(\sin\theta+|\!\IM \alpha|)^2\bigr)\sin^2\theta\dif \theta\\
        =&\frac{1}{\pi}\int_0^{\pi}\log(1+|\alpha|^2-2|\alpha|\cos(\theta-\phi))\sin^2\theta\dif \theta.
    \end{align*}
    The formula follows immediately from Lemma \ref{lem:jenformula}.
\end{proof}

The method herein can be applied to arbitrary one-slice preserving polynomial.

\begin{thm}
    Let $P(x)$ be a slice regular polynomial which preserves $\C_{I}$. Then as $n\to+\infty$ one has
    \[\lim_{n\to+\infty}\frac{1}{n}\mH(P^{*n})=\frac{2}{\pi}\int_{0}^{\pi}\log\max\{\abs{P(e^{I\theta})},\abs{P(e^{-I\theta})}\}\sin^2\theta\dif\theta,\]
and as $n\to-\infty$ one has,
\[\lim_{n\to-\infty}\frac{1}{n}\mH(P^{*n})=\frac{2}{\pi}\int_{0}^{\pi}\log\min\{\abs{P(e^{I\theta})},\abs{P(e^{-I\theta})}\}\sin^2\theta\dif\theta.\]
\end{thm}
\begin{proof}
    If $n\geq 0$, the proof is exactly the same as the proof of Theorem~\ref{thm:iterateda}. If $n<0$, like Proposition~\ref{prop:reciprocali}, we get
 \[\mH(P^{*-n})=\mH((P^{c})^{*n})-n\mH(P^{s}).\]
 The conjugate $P^{c}$ and symmetrization $P^{s}$ both preserve $\C_{I}$, and
 \[P^{c}(e^{I\theta})=\overline{P(e^{-I\theta})},\quad P^{s}(e^{I\theta})=P(e^{I\theta})\overline{P(e^{-I\theta})}.\]
These combine to give us the formula when $n<0$.
\end{proof}

Once more, we offer the following Mahler measures for one-slice preserving polynomial.

\begin{prop}
    We have
    \[\mH((x-i)\ast(x-2i))=\frac{9}{16}+\log 2.\]
\end{prop}
\begin{proof}
    By Proposition~\ref{prop:pxtheta}, we find that
\begin{align*}
   \mH((x&-i)\ast(x-2i))\\
    &=\frac{1}{2\pi}\int_{0}^{\pi}\int_{0}^{\pi}\log \left[2(5+4\sin^2\theta-9\sin\theta\cos\theta_1)\right]\sin^2\theta\sin\theta_{1}\dif\theta\dif\theta_{1}.
\end{align*}
Using~\eqref{eq:logabtheta} we get
\begin{align*}
H&(r)\coloneqq\int_{0}^{\pi}\log \big(5+4r^2-9r\cos\theta_1\big)\sin\theta_{1}\dif\theta_1\\
=\frac{1}{9r}&(4r^2+9r+5)\log (r+1)(4r+5)-\frac{1}{9r}(4r^2-9r+5)\log (r-1)(4r-5)-2,
\end{align*}
where $r=\sin\theta$. For the next step, we mention that $H(r)=H(-r)$. This implies that
\begin{align*}
  \mH&((x-i)\ast(x-2i))=\frac{1}{4\pi}\int_{0}^{2\pi}H(\cos\theta)\cos^2\theta\dif\theta+\frac{1}{2}\log 2.
\end{align*}
The integral of $H$ is in fact $\frac{9}{16}+\frac{1}{2}\log 2$, which can be computed by
\begin{align*}
\frac{1}{4\pi}&\int_{0}^{2\pi}H(\cos\theta)\cos^2\theta\dif\theta\\
  &=\frac{1}{18}(4A^{+}(3)+9A^{+}(2)+5A^{+}(1)+4B^{+}(3)+9B^{+}(2)+5B^{+}(1))-\frac{1}{2}\\
  &-\frac{1}{18}(4A^{-}(3)-9A^{-}(2)+5A^{-}(1)+4B^{-}(3)-9B^{-}(2)+5B^{-}(1)),
\end{align*}
where $A^{\pm}(n)$ and $B^{\pm}(n)$ are given by the following formulas
\begin{align*}
   A^{\pm}(n)\coloneqq\frac{1}{2\pi}\int_{0}^{2\pi}\log(1\pm\cos\theta)\cos^n\theta\dif\theta=
    \begin{cases}
      \pm\frac{5}{6}\phantom{-\frac{1}{2}\log 2}&n=3,\\
      -\frac{1}{4}-\frac{1}{2}\log 2&n=2,\\
      \pm 1\phantom{-\frac{1}{2}\log 2}&n=1,
    \end{cases}
\end{align*}
and
\begin{align*}
    B^{\pm}(n)\coloneqq\frac{1}{2\pi}\int_{0}^{2\pi}\log(5\pm4\cos\theta)\cos^n\theta\dif\theta=
    \begin{cases}
      \pm\frac{37}{96}\phantom{+\log 2}&n=3,\\
      -\frac{1}{16}+\log 2&n=2,\\
      \pm \frac{1}{2}\phantom{+\log 2}&n=1.
    \end{cases}
\end{align*}
One can use Cauchy's theorem to deduce these formulas, for example,
\begin{align*}
    \frac{1}{2\pi}\int_{0}^{2\pi}\log(5+4\cos\theta)\cos^3\theta&\dif\theta=\frac{1}{4}\IM\oint_{|z|=1}\log(z+2)\Big(z+\frac{1}{z}\Big)^3\,\frac{\dif z}{z}\\
    =&\frac{\pi i}{2}\Res_{z=0}\frac{\log(z+2)}{z}\Big(z+\frac{1}{z}\Big)^3=\frac{37}{96}.\qedhere
\end{align*}
\end{proof}

Finally, we provide an example of the Mahler measure of a slice regular polynomial that does not preserve any slice.

\begin{prop}
    We have
    \[\mH((x-i)*(x-j))=\frac{\pi}{2}-1.\]
\end{prop}
\begin{proof}
Let us write $x$ in the hyperspherical form
\[x=\sin \theta\sin\theta_1\cos \theta_2  +k\sin \theta\sin\theta_1\sin\theta_2 +i\cos\theta+j\sin\theta\cos \theta_1 .\]
In this representation, the Mahler measure takes the form of an integral
\begin{align*}
 \mH((x-i)&*(x-j))\\
 &=\frac{1}{4\pi^2}\int_{0}^{2\pi}\int_{0}^{\pi}\int_0^{\pi}\log A(\theta,\theta_1,\theta_2)\sin^2\theta\sin\theta_1\dif \theta\dif \theta_1\dif \theta_2,
\end{align*}
where
\[A(\theta,\theta_1,\theta_2)=A_{1}\cdot A_{2}=4(1-\cos\theta)\cdot\Big(1+\cos^2\frac{\theta}{2}\sin^2\theta_{1}\sin 2\theta_{2}\Big).\]

To compute it, we need to evaluate two separate integrals. The first integral, involving $A_1$,  can be computed using Lemma~\ref{lem:jenformula}
\begin{align*}
    \int_{S^{3}}&\log A_1\dif\mu(S^{3})=\frac{2}{\pi}\int_{0}^{\pi}\log(4-4\cos\theta)\sin^2\theta\dif\theta=\frac{1}{2}+\log 2.
\end{align*}

The second integral, involving $A_2$, requires expanding the $\log$ as
\[\log A_2=\sum_{n=1}^{\infty}(-1)^{n-1}\frac{1}{n}\cos^{2n}\frac{\theta}{2}\sin^{2n}\theta_{1}\sin^{n} 2\theta_{2}.\]
It is not hard to find
\[\int_{0}^{\pi}\cos^{2n}\frac{\theta}{2}\sin^2\theta\dif\theta=2\sqrt{\pi}\,\frac{\Gamma(n+\frac{3}{2})}{\Gamma(n+3)},\]
\[\int_{0}^{\pi}\sin^{2n+1}\theta_1\dif\theta_1=\sqrt{\pi}\,\frac{\Gamma(n+1)}{\Gamma(n+\frac{3}{2})},\]
and
\[\int_{0}^{2\pi}\sin^{n}2\theta_2\dif\theta_2=2\sqrt{\pi}\frac{\Gamma(\frac{n+1}{2})}{\Gamma(\frac{n}{2}+1)}\quad\text{if $n$ is even, otherwise } 0.\]
Integrating term by term, we obtain
\[
    \int_{S^{3}}\log A_2\dif\mu(S^{3})=-\frac{1}{4}\sum_{m=1}^{\infty}\frac{\left(\frac{1}{2}\right)_m}{m\!\left(m+\frac{1}{2}\right)\!(m+1)!}.
\]

Finally, we simplify our task to proving a specific identity
\[S\coloneqq\frac{1}{4}\sum_{m=1}^{\infty}\frac{\left(\frac{1}{2}\right)_m}{m\!\left(m+\frac{1}{2}\right)\!(m+1)!}=\frac{5}{2}+\log 2-\pi.\]
Observe first that the expression
\[\frac{1}{4}\sum_{m=1}^{\infty}\frac{\left(\frac{1}{2}\right)_m}{\!\left(m+\frac{1}{2}\right)\!(m+1)!}x^m=\sqrt{x}\arcsin \sqrt{x}+(1-x)^{\frac{1}{2}}-1,\]
implies that
\[S_{1}\coloneqq\frac{1}{4}\sum_{m=1}^{\infty}\frac{\left(\frac{1}{2}\right)_m}{\!\left(m+\frac{1}{2}\right)\!(m+1)!}=\frac{\pi}{2}-\frac{3}{2}.\]
Meanwhile, we find that
\begin{align*}
  S_{2}\coloneqq\frac{1}{4}\sum_{m=1}^{\infty}\frac{\left(\frac{1}{2}\right)_m}{m(m+1)!}&=\int_{0}^{1}\left(\frac{1-\sqrt{1-x}}{2x^2}-\frac{1}{4x}\right)\dif x=\frac{1}{2}\log 2-\frac{1}{4}.
\end{align*}
Consequently, the sum $S$ is
\[S=2(S_{2}-S_{1})=\frac{5}{2}+\log 2-\pi,\]
thus establishing the desired identity.
\end{proof}
\section{Examples of non-commutative linear polynomials}\label{sec6}

In this section, we give some examples of Mahler measures of non-commutative linear polynomials. 

\begin{defn}
    Let $L:\mbh\to\mbh$ be an operator on quaternions. The operator is said to be \emph{real-linear} if for any $x,y\in\mbh$,
\[L(x+y)=L(x)+L(y),\]
\[L(qx)=qL(x),\quad\text{for }q\in\R.\]
\end{defn}
A well-known fact is that every real-linear operator can be depicted as a non-commutative linear polynomial
\[L(x)=\sum_{i=1}^{n} p_{i}xq_{i},\quad\text{where }p_{i},q_{i}\in\mbh\text{ are non-zero}.\]
For example, the quaternionic conjugation can be written as
\[\overline{q}=-\frac{1}{2}(q+iqi+jqj+kqk).\]
Let $x=x_1+x_2i+x_3j+x_4k$ and $v=(x_1,x_2,x_3,x_4)^T$, then there exists a real symmetric $(4\times4)$-matrix $P$ such that
\begin{equation}\label{eq:vpv}
    |L(x)|^2=v^TPv.
\end{equation}

Let $p_1,p_2,q_1,q_2\in \mbh$ be non-zero quaternions. Within this section, we examine a particular type of linear polynomials that can be written as a sum of two monomials
\[L(x)=p_1xq_1+p_2xq_2.\] 
Upon direct simplification, we find
\[\mH(L)=\log \abs{p_2}+\log \abs{q_1}+\mH(p_{2}^{-1}p_{1}x+xq_{2}q_{1}^{-1}).\]
Thus, our attention is directed towards those kinds of polynomials under the Sylvester form $L(x)=ax+xb$.

Before the proof of Theorem~\ref{thm:ab}, we prepare a simple lemma.
\begin{lem}\label{lem:eigenofx}
    Let $a,b\in \mbh$ be non-zero and let $x=x_1+x_2i+x_3j+x_4k$. Then the four eigenvalues of the matrix $P$ in Eq. \eqref{eq:vpv} are 
    \begin{align*}
        \lambda_1^{+}&=\lambda_2^{+}=(\RE a+\RE b)^2+(|\!\IM a|+|\!\IM b|)^2,\\  \lambda_1^{-}&=\lambda_2^{-}=(\RE a+\RE b)^2+(|\!\IM a|-|\!\IM b|)^2.
    \end{align*}
\end{lem}

\begin{proof}
 The eigenvalues can be calculated through direct computation. Refer, for example, to \cite[Theorem 2.3.3]{Rod}.
\end{proof}

\begin{proof}[Proof of Theorem~\ref{thm:ab}]
    Let $v=(x_1,x_2,x_3,x_4)^T$. Then Lemma~\ref{lem:eigenofx} shows that there exists an orthogonal matrix $Q$ such that 
    \[|ax+xb|^2=v^TQ^T\Lambda Qv,\]
    where $\Lambda=\diag(\lambda_1^{+},\lambda_2^{+},\lambda_1^{-},\lambda_2^{-})$. Put $u=Qv$. By the orthogonal invariance of Haar measure on $S^3$, the vector $u$ is again uniformly distributed on $S^3$. Therefore
    \[
    |ax+xb|^2=\lambda^+(u_1^2+u_2^2)+\lambda^-(u_3^2+u_4^2).
    \]
    We now write
    \[
    u=(\sqrt{1-t}\cos\varphi,\sqrt{1-t}\sin\varphi,\sqrt{t}\cos\psi,\sqrt{t}\sin\psi),
    \]
    where $0\leq t\leq 1$ and $0\leq\varphi,\psi<2\pi$. In these coordinates the probability Haar measure is
    \[
    \dif\mu_{S^3}=\frac{1}{4\pi^2}\dif t\dif\varphi\dif\psi.
    \]
    Thus $t=u_3^2+u_4^2$ is uniformly distributed on $[0,1]$. It follows that
    \begin{align*}
        m_{\mbh}&(L)=\int_{\mbt^1(\mbh)}\log|ax+xb|\dif \mu(x)\\
        &=\frac{1}{4\pi^2}\int_0^1\int_0^{2\pi}\int_0^{2\pi}\frac{1}{2}\log\bigl(\lambda^+(1-t)+\lambda^-t\bigr)\dif\varphi\dif\psi\dif t\\
        &=\frac{1}{2}\int_0^1\log\bigl(\lambda^+(1-t)+\lambda^-t\bigr)\dif t.
    \end{align*}
If $\lambda^+\ne\lambda^-$, the last integral gives   
\[
        m_{\mbh}(L)=\frac{1}{2}\left(\frac{\lambda^{+}\log\lambda^{+}-\lambda^{-}\log \lambda^{-}}{\lambda^{+}-\lambda^{-}}-1\right).
\]
If $\lambda^+=\lambda^-$, then $L$ is not identically zero only when $\lambda^+=\lambda^->0$, and the same integral equals $\frac{1}{2}\log\lambda^+$, which is the limiting value of the displayed expression.
\end{proof}

Here we also exhibit a special case when $a=-b$. In particular, this gives the quaternionic Mahler measure of the commutator $[a,x]=ax-xa$.
\begin{thm}\label{thm:a=b}
    Let $a\in \mbh$ and $c\in \mbr$ be non-zero. Let $L(x)=ax-xa+c$.     If $|\!\IM a|\ne0$, then
    \[m_{\mbh}(L)=\log |\!\IM a|+\frac{1}{8}\left((4+c_{0}^2)\log(4+c_{0}^2)-2c_{0}^2\log c_{0}-4\right),\]
    where $c_{0}=|c|/|\!\IM a|$.
\end{thm}

\begin{proof}
    Let us put $a=\RE a+I|\!\IM a|$. Then we get
    \begin{align*}
        m_{\mbh}(ax&-xa+c)=m_{\mbh}(I|\!\IM a|x-xI|\!\IM a|+c)\\
       &=\log |\!\IM a|+m_{\mbh}(Ix-xI+c_{0}).
    \end{align*}
    
Observe that $I$ commutes with $\mbc_{I}$, then letting $I,J,K=IJ$ be an orthogonal basis of $\S$ we have 
\[\log|I(z+Jw)-(z+Jw)I+c_{0}|=\log|2Kw+c_{0}|.\]
Let $w=u-Iv$, then the integral can be simplified to an integral on the unit disk 
\begin{align*}
    m_{\mbh}(Ix-xI+c_{0})&=\frac{1}{4\pi^2}\int_{|z|^2+|w|^2=1}\log|2Kw+c_{0}|\frac{\dif z}{z}\dif w\dif \overline{w}\\
    &=\frac{i}{2\pi}\int_{0\leq |w|\leq 1}\log|2K w+c_{0}|\dif w\dif \overline{w}.
\end{align*}
Expanding this expression explicitly, we get
\begin{align*}
    m_{\mbh}(Ix-xI+c_{0})&=\frac{1}{2\pi}\int_{u^2+v^2\leq 1}\log\left(c_{0}^2+4u^2+4v^2\right)\dif u\dif v\\
   &=\frac{1}{8}\left((4+c_{0}^2)\log(4+c_{0}^2)-2c_{0}^2\log c_{0}-4\right).\qedhere
\end{align*}
\end{proof}

\section{Multivariable quaternionic Mahler measures}\label{sec7}

This section is dedicated to providing some examples of multivariable quaternionic Mahler measures. Recall that for a non-commutative multivariable quaternionic polynomial $P(x_1,\dots,x_n)$, its Mahler measure is the integral
\[\mH(P)=\int_{\mbt^n(\mbh)}\log\abs{P(x_1,\dots,x_n)}\dif\mu(\mbt^n(\mbh)).\]

The theoretical framework concerning slice regular functions of multiple quaternionic variables is elucidated in Ghiloni--Perotti~\cite{GP2}. Multivariate slice regular polynomials, as a special kind of slice regular functions, are polynomial  functions $P(x_1,\dots,x_n)$ on $\mbh^n$ of the form
\[P(x_1,\dots,x_n)=\sum_{0\leq \ell_1+\cdots+\ell_n\leq \deg(P)} x_{1}^{\ell_1}\dots x_{n}^{\ell_n}a_{\ell_1,\dots,\ell_n},\] 
where $a_{\ell_1,\dots,\ell_n}\in\mbh$. In the recent papers~\cite{GSV1, GSV2}, Gori, Sarfatti and Vlacci study the zero sets and Hilbert Nullstellensatz of slice regular polynomials in several variables. 

In the theorem presented below, we confirm the existence of Mahler measures for slice regular polynomials of two quaternion variables.
\begin{thm}
    Let $P(x,y)$ be a slice regular quaternionic polynomial of two variables which is not identically zero on $\mbt^2(\mbh)$. Then its Mahler measure $\mH(P)$ exists as an improper Riemann integral.
\end{thm}
\begin{proof}
    As in the proof of Theorem~\ref{thm:existence}, we need to establish the estimate
    \[\mu_{\mbh}\left(\big\{(x,y)\in\mbt^{2}(\mbh)\,\big{|}\,\abs{P(x,y)}\leq\varepsilon\big\}\right)\leq C\varepsilon^{\delta}.\]
   
Let $Q(x,y)$ be the following one-variable symmetrization of $P(x,y)$ with respect to the variable $x$,
\[
Q(x,y)=\sum_{\ell}x^{\ell}\sum_{h}\sum_{h_1,h_2}y^{h_1}\,\overline{y}^{h_2}a_{h,h_1}\overline{a}_{\ell-h,h_2}.
\]
Since $P(x,y)$ is left slice regular in $x$, using the Cauchy--Schwarz estimate as in Lemma~\ref{lem:measurebound}, it is enough to settle  
\[\mu_{\mbh}\left(\big\{(x,y)\in\mbt^{2}(\mbh)\,\big{|}\,\abs{Q(x,y)}\leq\varepsilon\big\}\right)\leq C\varepsilon^{\delta}.\]
The function $Q(x,y)$ is a real polynomial with respect to the variable $x$. We start by induction on the degree $n$ of $Q(x,y)$ in $x$. We may factor $Q(x,y)$ by
\[Q(x,y)=A(y,\overline{y})\prod_{\ell=1}^{m_1}(x^2-2\alpha_{\ell}(y,\overline{y})x+\beta_{\ell}(y,\overline{y}))\prod_{h=1}^{m_2}(x-\gamma_{h}(y,\overline{y})),\]
where $A,\alpha_{\ell},\beta_{\ell}$ and $\gamma_{h}$ are real-valued algebraic functions with respect to $y$ and $\overline{y}$. Note that these functions rely only on $\RE y$ and $|\!\IM y|$, and do not depend on the imaginary unit within $y=\RE y+I|\!\IM y|$. Thus, if $|Q(x,y)|<\varepsilon$, then either
\[\abs{A(y,\overline{y})}\leq \varepsilon^{\frac{1}{n+1}},\]
or one of the real factors
\[
\abs{x^2-2\alpha_{\ell}(y,\overline{y})x+\beta_{\ell}(y,\overline{y})}\quad\text{or}\quad |x-\gamma_{h}(y,\overline{y})|\leq \varepsilon^{\frac{1}{n+1}}.
\]

In the first case, $A(y,\overline{y})=A(y,y^{-1})$ is a Laurent polynomial that relies only on $\RE y$ and $|\!\IM y|$. So we can deduce that
\[
\mu_{\mbh}\left(\big\{y\in\mbt^{1}(\mbh)\,\big{|}\,\abs{A(y,y^{-1})}\leq\varepsilon\big\}\right)\leq 2\mu_{\C}\left(\big\{y\in\mbt^{1}(\C)\,\big{|}\,\abs{A(y,y^{-1})}\leq\varepsilon\big\}\right).
\]
The latter is thusly controlled by $C_1\varepsilon^{\delta_1}$ with the help of Lemma~\ref{lem:EW}.

In the latter cases, by the same fashion, we have
\begin{align*}
    \mu_{\mbh}\big(\big\{(x,y)&\in\mbt^{2}(\mbh)\,\big{|}\,\abs{x-\gamma_{h}(y,\overline{y})}\leq\varepsilon\big\}\big)\\
    &\leq 4\mu_{\C}\left(\big\{(x,y)\in\mbt^{2}(\C)\,\big{|}\,\abs{x-\gamma_{h}(y,y^{-1})}\leq\varepsilon\big\}\right).
\end{align*}
For fixed $y\in\C$, using again Lemma~\ref{lem:EW} for $P(x)=x-\gamma_{h}(y,y^{-1})$, it is apparent that
\[\mu_{\C}\left(\big\{x\in\mbt^{1}(\C)\,\big{|}\,\abs{x-\gamma_{h}(y,y^{-1})}\leq\varepsilon\big\}\right)\leq C_2\varepsilon^{\delta_2}.\]
This makes us realize that
\[\mu_{\mbh}\left(\big\{(x,y)\in\mbt^{2}(\mbh)\,\big{|}\,\abs{x-\gamma_{h}(y,y^{-1})}\leq\varepsilon\big\}\right)\leq C_2\varepsilon^{\delta_2}.\]
For the quadratic factor $x^2-2\alpha_{\ell}(y,\overline{y})x+\beta_{\ell}(y,\overline{y})$, using the same approach yields the same bound. 

Putting everything together, we finally conclude that there exist  $C,\delta>0$ so that
\[\mu_{\mbh}\left(\big\{(x,y)\in\mbt^{2}(\mbh)\,\big{|}\,\abs{P(x,y)}\leq\varepsilon\big\}\right)\leq C\varepsilon^{\delta}.\qedhere\]
\end{proof}

\begin{rmk}
The power bounds for small-value sets established in Lemma~\ref{lem:measurebound} and in the proof above can also be obtained from general polynomial sublevel estimates. For instance, after writing the relevant quaternionic polynomial in real coordinates, one may apply the Carbery--Wright inequality~\cite{CW} to obtain a bound of the form $\mu\{|P|\leq\varepsilon\}\leq C\varepsilon^{\delta}$. A general discussion of this approach will be given in future work.
\end{rmk}

% \begin{rmk}
% In~\cite{GSV1},  Gori--Sarfatti--Vlacci offer the following example, the slice regular polynomial $P(x,y)=xy-k$ has its \emph{multivariable symmetrization} \[P^{s}(x,y)=x^2y^2+1.\] However, it holds $P(i,j)=0$ while $P^{s}(i,j)=2$. It can be seen that the zero of a polynomial $P$ might no longer be the zero of its symmetrization, rendering the proof in one variable ineffective in the multivariable case. 
% \end{rmk}

Subsequently, we compute the quaternionic Mahler measure $\mH(1+x+y)$.
\begin{proof}[Proof of Theorem \ref{thm:Smyth}]
By changing the variable $y\to yx$, we get $m_{\mbh}(x+y+1)=m_{\mbh}(x+yx+1)$. By definition, we have 
\begin{align*}
    m_{\mbh}((1+y)x+1)&=\int_{\mbt^2(\mbh)}\log|(1+y)x+1|\dif \mu(x)\dif \mu(y)\\
    &=\frac{2}{\pi}\int_0^{\pi}\int_{\mbt^1(\mbh)}\log|(1+e^{i\theta})x+1|\sin^2\theta\dif \mu(x)\dif \theta.
\end{align*}
Since $|1+e^{i\theta}|\geq 1$ if and only if $\theta\in [0,\frac{2}{3}\pi]$, by Lemma~\ref{lem:jenformula}, we have 
\begin{align*}
    &m_{\mbh}(x+y+1)\\
    =&\frac{1}{2\pi}\int_{\frac{2}{3}\pi}^{\pi}|1+e^{i\theta}|^2\sin^2\theta\dif \theta+\frac{1}{2\pi}\int_0^{\frac{2}{3}\pi}\frac{\sin^2\theta}{|1+e^{i\theta}|^{2}}\dif \theta+\frac{2}{\pi}\int_0^{\frac{2}{3}\pi}\log|1+e^{i\theta}|\sin^2\theta\dif \theta.
\end{align*}
Denote the above integrals by $I_1,I_2$ and $I_3$ respectively. The integrals $I_1,I_2$ are easy to calculate via $|1+e^{i\theta}|^2=2+2\cos\theta$. A direct calculation gives
\[I_1=\frac{1}{6}-\frac{\sqrt{3}}{4\pi}\quad\text{and}\quad I_2=\frac{1}{6}-\frac{\sqrt{3}}{8\pi}.\]
For the third integral, we have 
\[I_3=\frac{2}{\pi}\RE \int_0^{\frac{2}{3}\pi}\log(1+e^{i\theta})\sin^2\theta\dif\theta=-\frac{1}{2\pi i}\IM\int_1^{\rho}\log(1+z)\Big(z-\frac{1}{z}\Big)^2\,\frac{\dif z}{z},\]
where $\rho=e^{\frac{2}{3}\pi i}$. Since
\begin{align*}
    \int \Big(z+&\frac{1}{z^3}\Big)\log(1+z)\dif z\\
    &=\frac{1}{4z^2}\big(2(z^4-1)\log(1+z)-2z^2\log(z)-2z+2z^3-z^4\big),
\end{align*}
it can be deduced that
\[\IM\int_1^{\rho} \Big(z+\frac{1}{z^3}\Big)\log(1+z)\dif z=\Big(\frac{5\sqrt{3}}{8}-\frac{\pi}{3}\Big)i.\]
Finally, for the integral of $\log(1+z)/z$, by expanding $\log(1+z)$ in a Taylor series, we have
\begin{align*}
    \IM&\int_1^{\rho}\frac{\log(1+z)}{z}\dif z=\IM\int_1^{\rho}\sum_{n\geq 1}\frac{(-1)^{n-1}z^{n-1}}{n}\dif z\\
    &=\sum_{n\geq 1}\frac{(-1)^{n-1}i\sin \frac{2n\pi}{3}}{n^2}=\frac{3\sqrt{3}\,i}{4}L(\chi_{-3},2).
\end{align*}
Combining these results together, we get the desired conclusion.
\end{proof}

More generally, we can consider the Mahler measure of $ax+by+c$. If $|a|,|b|,|c|$ form a triangle, the above calculation still works, yet a general formula would be of much difficulty. If $|a|,|b|,|c|$ do not form a triangle, following the same argument as Maillot~\cite{Mai}, we are able to give the proof of~\ref{thm:Maillot}.

\begin{proof}[Proof of Theorem~\ref{thm:Maillot}]
    By multiplying the polynomial on the left by a unit quaternion and changing the variables $x$ and $y$ by unit factors, we may assume that $a,b,c\in\mbr^+$. Without loss of generality, suppose that $a\geq b\geq c$. By our assumption, $|a|,|b|,|c|$ do not form a triangle, so for any $y\in \mbt^1(\mbh)$, we have $|a+by|\geq a-b\geq c$. For a fixed $y\in \mbt^1(\mbh)$, by Theorem \ref{thm:Jensenformulas}, we have 
    \[\int_{\mbt^1(\mbh)}\log|(a+by)x+c|\dif \mu(x)=\log|a+by|+\frac{c^2}{4|a+by|^2}.\]
    So we get
    \begin{align*}
        m_{\mbh}(ax+by+c)&=\int_{\mbt^1(\mbh)}\int_{\mbt^1(\mbh)}\log|ax+byx+c|\dif \mu(x)\dif \mu(y)\\
        &=\int_{\mbt^1(\mbh)}\log|a+by|+\frac{c^2}{4|a+by|^2}\dif \mu(y)\\
        &=\frac{2}{\pi}\int_0^{\pi}\left(\log|a+be^{i\theta}|+\frac{c^2}{4|a+be^{i\theta}|^2}\right)\sin^2\theta\dif \theta.
    \end{align*}
    By applying Lemma \ref{lem:jenformula}, the first term equals $\log a+\frac{b^2}{4a^2}$. Following from Theorem \ref{thm:zetameaofx}, the second term equals $\frac{c^2}{4a^2}$. These combine to complete the formula.
\end{proof}

In the following theorem, we illustrate an example of the Mahler measure for a slice regular polynomial in two variables.

\begin{thm}\label{thm:x2-y2} We have the identity
\[\mH(x^2-y^2)=0.\]
\end{thm}
\begin{proof}
The Mahler measure can be rewritten as
\[\mH(x^2-y^2)=\mH((xy)^2-y^2)=\mH(xy-yx^{-1}).\]
Applying Theorem~\ref{thm:ab} in the variable $y$ yields
\[
    \mH(xy-yx^{-1})=\int_{\mbt^{1}(\mbh)}\left(\log(2|\!\IM x|)-\frac{1}{2}\right)\dif\mu(x).
\]
Observe that the integrand depends only on $|\!\IM x|$. Following the same pattern as Proposition~\ref{prop:pxtheta}, we get
\[
\int_{\mbt^{1}(\mbh)} \log{|\!\IM x|}\dif\mu(x)=\frac{2}{\pi}\int_{0}^{\pi}\log(\sin\theta)\sin^2\theta\dif\theta.
\]
The latter integral can be computed in the same manner as Lemma~\ref{lem:jenformula}. Via the expansion
\[\log(2\sin\theta)=-\sum_{n=1}^{\infty}\frac{\cos 2n\theta}{n},\]
one finds that
\begin{align*}
\frac{2}{\pi}\int_{0}^{\pi}\log(\sin\theta)\sin^2\theta\dif\theta&=\frac{2}{\pi}\int_{0}^{\pi}\log(2\sin\theta)\sin^2\theta\dif\theta-\log 2\\
&=\frac{1}{2}-\log 2.
\end{align*}
Thus $\mH(x^2-y^2)=0$.
\end{proof}
We now give two examples of non-slice-regular polynomials in two variables.

\begin{thm}
   We have the following identities
   \[\mH(xy-yx)=0,\]
   \[\mH(xy-yx+1)=\frac{1}{4}\Big(1-\sqrt{5}+3\log\frac{3+\sqrt{5}}{2}\Big).\]
\end{thm}
\begin{proof}
    Using Theorem~\ref{thm:a=b} in the variable $y$, we get
\[
    \mH(xy-yx)=\int_{\mbt^{1}(\mbh)} \log{|\!\IM x|}\dif\mu(x)+\log 2-\frac{1}{2}.
\]
Recall the proof of Theorem~\ref{thm:x2-y2} where we get
\[
\int_{\mbt^{1}(\mbh)}\log{|\!\IM x|}\dif\mu(x)=\frac{1}{2}-\log 2.
\]
This proves the given identity. Similarly, the Mahler measure of $xy-yx+1$ can be reduced to
\[
     \mH(xy-yx+1)=\frac{1}{2}-\log 2+\frac{1}{\pi}\int_{0}^{\pi}M(\theta)\dif\theta,
\]
where
\[
M(\theta)=\Big(\frac{1}{4}+\sin^2\theta\Big)\log(5-4\cos^2\theta)-2\log(\sin\theta)\sin^2\theta-\sin^2\theta.
\]
The calculation uses exactly the same trick as Theorem~\ref{thm:x2-y2} and Lemma~\ref{lem:jenformula} before, so we omit the details.
\end{proof}

\section{Quaternionic Lehmer problem}\label{sec8}

In~\cite{Leh}, Lehmer asked what is the smallest non-zero Mahler measure of monic integer polynomials. In addition to his question, he gave the following example of degree $10$, irreducible reciprocal polynomials
\begin{align*}
    L_{10}(x)=x^{10}+x^9-x^7-x^6-x^5-x^4-x^3+x+1,
\end{align*}
where we have $\mC(L_{10})=0.16235\dots$ which is conjectured to be the smallest non-zero Mahler measure.

A parallel question is: \emph{what is the smallest non-zero quaternionic Mahler measure that a monic irreducible integer polynomial $f$ can achieve}? The requirement of irreducibility here is crucial. Since the quaternionic Mahler measures may not always be positive as shown by the following example
\[\mH((x^2+1)^n)=-\frac{1}{2}n\to-\infty\quad\text{as } n\to\infty.\]

\begin{conj}[Quaternionic Lehmer conjecture]\label{conj:Mahler}
    There is a constant $c$ such that for every monic, irreducible polynomial $P(x)\in \mbz[x]$, we have 
    \[m_{\mbh}(P)\geq c.\]
\end{conj}

One starts by observing the quaternionic Mahler measure of cyclotomic polynomials.
\begin{prop}
    Let $\Phi_n(x)$ be the $n$-th cyclotomic polynomial. Then we have 
    \[m_{\mbh}(\Phi_n)=\begin{cases}
        \frac{1}{4}\mu(n)&\text{if $2\nmid n$}\\
        \frac{1}{4}\mu(n/2)&\text{if $2\mid n$, $4\nmid n$}\\
        \frac{1}{2}\mu(n/2)&\text{if $4\mid n$},
    \end{cases}\]
    where $\mu(n)$ is the M\"obius function.
\end{prop}
\begin{proof}
    Since $\Phi_n(x)=\prod_{(d,n)=1}(x-\zeta_n^d)$, by Theorem \ref{thm:Jensenformulas}, we have 
    \[m_{\mbh}(\Phi_n)=\frac{1}{2}\sum_{(d,n)=1}m_{\mbh}((x-\zeta_n^d)(x-\zeta_n^{-d}))=\frac{1}{4}\sum_{(d,n)=1}\cos \frac{4d\pi}{n}.\]
    The right-hand side is the well-known Ramanujan's sum, see for instance \cite[Theorem 272]{HW}.
\end{proof}

The following proposition offers a trivial bound for quaternionic Mahler measure.

\begin{prop}~\label{prop:lowerbound}
If $P(x)$ is a monic slice-preserving polynomial of degree $n$, then we have
    \[\mH(P)\geq -\frac{1}{2}\left\lfloor\dfrac{n}{2}\right\rfloor,\]
where the equality holds if and only if $P(x)=(x^2+1)^m$ when $n=2m$ is even, and $P(x)=x(x^2+1)^m$ when $n=2m+1$ is odd.
\end{prop}
\begin{proof}
    With (1) and (2) in Theorem~\ref{thm:Jensenformulas}, one finds that the minimal Mahler measure for a linear polynomial is achieved when 
    \[\mH(x-\alpha)\geq \mH(x)=0, \]
    the minimal Mahler measure of a quadratic polynomial is achieved when 
    \[\mH(x^2-2\RE \alpha\cdot x+|\alpha|^2)\geq\mH(x^2+1)=-\frac{1}{2}.\]
Then using (3) of Theorem~\ref{thm:Jensenformulas} we get the lower bound of Mahler measures.
\end{proof}

Employing the Cauchy--Schwarz inequality as in Theorem~\ref{thm:existence}, we find the following lower bound for slice regular polynomials. It is natural to wonder if the lower bound in Proposition~\ref{prop:lowerbound} still holds when $n$ is odd.
\begin{cor}
    Let $P(x)$ be a monic slice regular quaternionic polynomial of degree $n$, then we have
    \[\mH(P)\geq-\frac{n}{4},\]
where the equality holds only if $n=2m$ is even and $P(x)=(x^2+1)^m$.
\end{cor}

The formula of Mahler measures simplifies considerably when all roots of the polynomial are located either outside or inside the unit circle.

\begin{prop}\label{cor:rootinunitcircle}
    Let $P(x)=a_nx^n+a_{n-1}x^{n-1}+\dots+a_0$ be a slice-preserving polynomial of degree $n\geq 1$, and set $a_j=0$ for $j<0$ or $j>n$. If all roots of $P$ lie inside the unit circle, then we have 
    \[m_{\mbh}(P)=\frac{a_{n-1}^2-2a_{n-2}a_n}{4a_n^2}+\log|a_n|.\]
    If all the roots of $P$ lie outside the unit circle, then we have 
    \[m_{\mbh}(P)=\frac{a_{1}^2-2a_{2}a_0}{4a_0^2}+\log|a_0|.\]
\end{prop}

\begin{proof}
    Note that the roots of $\widetilde{P}(x):=x^nP(1/x)$ are inside the unit circle if all roots of $P(x)$ are outside the unit circle, so we only need to prove the first assertion. Suppose $\beta_\ell$ is a non-real root of $P(x)$, then by Theorem \ref{thm:Jensenformulas}, we have 
    \[m_{\mbh}(x^2-2\RE \beta_\ell \cdot x+|\beta_\ell|^2)=\frac{1}{2}|\beta_\ell|^2\cos 2\phi_\ell=\frac{1}{4}(\beta_\ell^2+\ol{\beta}_\ell^2).\]
    So by factoring $P$ as in \eqref{eq:factorofp} and applying Theorem \ref{thm:Jensenformulas} again, we get 
    \[m_{\mbh}(P)=\log|a_n|+\frac{1}{4}\left(\sum_h\alpha_h^2+\sum_\ell(\beta_\ell^2+\ol{\beta}_\ell^2)\right)=\log|a_n|+\frac{a_{n-1}^2-2a_{n-2}a_n}{4a_n^2}.\qedhere\]
\end{proof}

As a direct consequence, this allows us to explore the quaternionic Mahler measure of a slice-preserving polynomial under the power change of the variable. While the classical Mahler measure is known to remain invariant under the transformation $x\mapsto x^r$, this property does not hold in our context. 
\begin{prop}
    Let $P(x)$ be a slice-preserving polynomial and $r>2$ be a positive integer, then 
    \[ m_{\mbh}(P(x^r))=m_{\mbc}(P).\]
\end{prop}
\begin{proof}
    Suppose $P(x)$ factors completely as in~\eqref{eq:factorofp}. Since the roots of $x^r-\alpha_h$ and $x^{2r}-2\RE \beta_\ell \cdot x^r+|\beta_\ell|^2$ are inside or outside the unit circle depending on whether $|\alpha_h| \leq 1$ or $|\alpha_h| \geq 1$, and similarly for $|\beta_\ell|$. So by Proposition \ref{cor:rootinunitcircle}, we have $m_{\mbh}(x^r-\alpha_h)=\log^+|\alpha_h|$ and $m_{\mbh}(x^{2r}-2\RE \beta_\ell \cdot x^r+|\beta_\ell|^2)=\log^+|\beta_\ell|^2$. Then the result is clear.
\end{proof}

Let $P(z)$ be a monic irreducible polynomial of degree $d$ with a complex factorization
\[P(z)=\prod_{\ell=1}^{d}(z-z_{\ell})=z^d+a_{d-1}z^{d-1}+\dots+a_0.\]
Numerical tests show that the polynomials with negative quaternionic Mahler measure are highly rare. The quaternionic Mahler measure of $P$ becomes much smaller when the roots $z_j$ are closer to the imaginary units $\pm i$ on the unit circle.  For an arc $I$ on the unit circle, let $N(I;P)$ denote the number of complex zeros $z_{j}$ with $\arg z_j\in I$. The \emph{discrepancy} of $P$ is defined by
\[\mathcal{D}(P)=\max_{I}\Big|N(I;P)-\frac{|I|}{2\pi}d\Big|.\]

Recall the famous result of Erd\"os--Tur\'an~\cite{ET} (see also Mignotte~\cite{Mig})  
\[\mathcal{D}(P)\leq \frac{8}{\pi}\sqrt{d h(P)},\]
where \[h(P)=\frac{1}{2\pi}\int_{0}^{2\pi}\log^{+}\frac{\abs{P(e^{i\theta})}}{\sqrt{|a_0|}}\dif\theta\quad\text{with }\log^{+}(x)=\max\{0,\log(x)\}.\]
Note the constant $8/\pi$ is given by Soundararajan~\cite{Sound}. This implies equidistribution of the root angles when the length of $P$ is not too large.

\begin{prop}~\label{prop:mh&mc}
    Let $P(z)$ be a monic irreducible polynomial with integer coefficients and non-zero constant term. Then we have
    \[\mH(P)\geq \mC(P)-\left(\frac{8}{\pi}+\frac{1}{2}\right)\sqrt{d h(P)}.\]
\end{prop}
\begin{proof}
    Write $z_{\ell}=r_{\ell}e^{i\phi_{\ell}}$ and put $\rho_{\ell}=\min\{r_{\ell}^{2},r_{\ell}^{-2}\}$. By~Theorem~\ref{thm:Jensenformulas}, we get
\[
    \mH(P)=\mC(P)+\frac{1}{4}\sum_{\ell=1}^{d}\rho_{\ell}\cos(2\phi_{\ell}).
\]
For $0<t<1$, set
\[
E_t^+=\{\theta\in[0,2\pi]:\cos 2\theta>t\},\quad
E_t^-=\{\theta\in[0,2\pi]:\cos 2\theta<-t\}.
\]
Viewed as subsets of the unit circle, each of $E_t^+$ and $E_t^-$ is a union of two arcs, and they have the same total length. Hence the discrepancy bound gives
\[
\left|N(E_t^\pm;P)-\frac{|E_t^\pm|}{2\pi}d\right|\leq 2\mathcal{D}(P),
\]
where $N(E;P)$ denotes the number of root arguments lying in $E$. Since $|E_t^+|=|E_t^-|$, it follows that
\[
\left|N(E_t^+;P)-N(E_t^-;P)\right|\leq 4\mathcal{D}(P).
\]
Moreover, for every $\theta$ we have
\[
\cos 2\theta=\int_0^1\left(\mathbf{1}_{E_t^+}(\theta)-\mathbf{1}_{E_t^-}(\theta)\right)\dif t.
\]
Applying this identity to $\theta=\phi_\ell$ and summing over all $\ell$ gives
\[
    \left|\sum_{\ell=1}^{d}\cos(2\phi_\ell)\right|\leq 4\mathcal{D}(P)\leq \frac{32}{\pi}\sqrt{d h(P)}.
\]
Moreover,
\[
    \frac{1}{4}\sum_{\ell=1}^{d}(1-\rho_\ell)\leq \min\left\{h(P),\frac{d}{4}\right\}\leq \frac{1}{2}\sqrt{d h(P)}.
\]
Here we used $1-e^{-2t}\leq 2t$ and the estimate
\[
    h(P)\geq \frac{1}{2}\sum_{\ell=1}^{d}\left|\log r_\ell\right|,
\]
which follows from Jensen's formula and the definition of $h(P)$. Therefore
\begin{align*}
    \mH(P)-\mC(P)
    &\geq -\frac{1}{4}\left|\sum_{\ell=1}^{d}\cos(2\phi_\ell)\right|
    -\frac{1}{4}\sum_{\ell=1}^{d}(1-\rho_\ell)\\
    &\geq -\left(\frac{8}{\pi}+\frac{1}{2}\right)\sqrt{d h(P)}.\qedhere
\end{align*}
\end{proof}

Following Hughes--Nikeghbali~\cite{HN}, we derive the following result in the same way as Proposition~\ref{prop:mh&mc}. 

\begin{prop}
Let $(P_d)$ be a sequence of complex polynomials, with
\[P_{d}(z)=\sum_{k=0}^{d}a_{d,k}z^k,\]
such that $a_{d,0}a_{d,d}\neq 0$ for all $d$, let
\[L_{d}(P_d)=\log\left(\sum_{k=0}^{d}|a_{d,k}|\right)-\frac{1}{2}\log |a_{d,0}|-\frac{1}{2}\log |a_{d,d}|.\]
If $L_{d}(P_d)=O(1)$, then we have 
\[\frac{1}{\sqrt{d}}\,\abs{\!\mH(P_{d})-\mC(P_{d})}=O(1)\qquad \text{as $d\to \infty$}.\]
\end{prop}

We now turn to some numerical evidence. All irreducibility assertions in these examples were verified in PARI/GP~\cite{PARI2} using \texttt{polisirreducible}.

\begin{exa}
By examining monic irreducible polynomials with integer coefficients of degree $\leq 10$ and coefficients in $[-3,3]$, we found that the minimal quaternionic Mahler measure is reached by the following irreducible polynomial of degree $8$:
\[
Q_{8}(x)=x^8-x^7+x^6-2x^5+x^4-2x^3+x^2-x+1.
\]
Its Mahler measure is
\[
m_{\mbh}(Q_8)=-0.26975\dotsb .
\]
This example is related to cyclotomic polynomials by
\[
Q_{8}(x)=\Phi_5(x)\Phi_{10}(x)-x\Phi_3(x)\Phi_4(x)\Phi_6(x).
\]
Meanwhile, the quaternionic Mahler measure of Lehmer's polynomial $L_{10}$ is
\[
m_{\mbh}(L_{10})=0.247131\dotsb .
\]
\end{exa}

\begin{exa}
Inspired by the observation of $Q_{8}$, by testing millions of combinations involving cyclotomic polynomials, we found the following two irreducible examples with Mahler measure less than $-0.5$. Let
\[
Q_{66}(x)=\prod_{h=3}^{10}\Phi_h(x)\Phi_{14}(x)\Phi_{18}(x)\Phi_{90}(x)+x\Phi_{204}(x).
\]
Then $Q_{66}$ is irreducible of degree $66$, and
\[
m_{\mbh}(Q_{66})=-0.5829\dotsb .
\]
Similarly, let
\[
Q_{50}(x)=\Phi_4(x)\Phi_5(x)\Phi_6(x)\Phi_7(x)\Phi_{10}(x)\Phi_{80}(x)+x\Phi_{140}(x).
\]
Then $Q_{50}$ is irreducible of degree $50$, and
\[
m_{\mbh}(Q_{50})=-0.5145\dotsb .
\]
It can thus be inferred that, if Conjecture~\ref{conj:Mahler} is true, then
\[
c\leq -0.5829\dotsb .
\]
\end{exa}

\end{document}